\documentclass[12pt,reqno]{amsart}
\usepackage{amsfonts,amssymb,amsmath}
\usepackage{ytableau}
\usepackage{nicematrix}
\usepackage{tikz}
\usepackage{xcolor}
\usepackage{enumerate}
\usepackage{multirow}
\usepackage{float}
\usepackage{mathtools}

\definecolor{fgreen}{RGB}{44,144, 14}

\renewenvironment{proof}{{\bfseries Proof.}}{\qed}

\numberwithin{equation}{section} 
\newtheorem{theorem}{Theorem}[section] 
\newtheorem{proposition}[theorem]{Proposition} 
\newtheorem{corollary}[theorem]{Corollary} 
\newtheorem{lemma}[theorem]{Lemma} 

\theoremstyle{definition}
\newtheorem{definition}[theorem]{Definition} 
\newtheorem{remark}[theorem]{Remark} 
\newtheorem{example}[theorem]{Example}

\usepackage[colorlinks=true,citecolor=cyan,  
 urlcolor=blue,linkcolor=blue]{hyperref}

\def\O{\mathbb O}

\def\C{\mathbb C}

\def\F{\mathbb F}

\def\d{\mathbf{ d}}

\def\N{\mathbb N}

\def\E{\mathbb E}

\def\d{\mathcal D}

\newcommand{\SL}{\mathrm{SL}}

\newcommand{\GL}{\mathrm{GL}}

\def\d{\mathcal D}

\def\C{\mathbb {C}}
\def\N{\mathbb {N}}

\def\O{\mathbb {O}}
\def\E{\mathbb {E}}
\def\F{\mathbb {F}}

\def\d{\mathbf {d}}

\def\GL{\rm GL}
\def\SL{\rm SL}

\newcommand{\defref}[1]{Definition~\ref{#1}}
\newcommand{\secref}[1]{Section~\ref{#1}}
\newcommand{\thmref}[1]{Theorem~\ref{#1}}
\newcommand{\lemref}[1]{Lemma~\ref{#1}}
\newcommand{\remref}[1]{Remark~\ref{#1}}
\newcommand{\propref}[1]{Proposition~\ref{#1}}

\begin{document}
\title[Strongly reversible classes in $\mathrm{SL}(n,\mathbb{C})$]{Strongly reversible classes in $\mathrm{SL}(n,\mathbb{C})$}

 \author[K. Gongopadhyay,    T. Lohan and  C. Maity]{Krishnendu Gongopadhyay, 
 Tejbir Lohan and Chandan Maity}

\address{Indian Institute of Science Education and Research (IISER) Mohali,
 Knowledge City,  Sector 81, S.A.S. Nagar 140306, Punjab, India}
\email{krishnendu@iisermohali.ac.in, krishnendug@gmail.com}

 \address{Indian Institute of Technology Kanpur, Kanpur 208016, Uttar Pradesh, India}
\email{tejbirlohan70@gmail.com}

\address{Indian Institute of Science Education and Research (IISER) Berhampur, Berhampur 760003, Odisha, India}
\email{maity.chandan1@gmail.com, cmaity@iiserbpr.ac.in}

\makeatletter
\@namedef{subjclassname@2020}{\textup{2020} Mathematics Subject Classification}
\makeatother


\begin{abstract} 
An element of a group is called \textit{strongly reversible} or \textit{strongly real}  if it can be expressed as a product of two involutions. We provide necessary and sufficient conditions for an element of $\mathrm{SL}(n,\mathbb{C})$ to be a product of two involutions. In particular, we classify the strongly reversible conjugacy classes in $\mathrm{SL}(n,\mathbb{C})$.
\end{abstract}
\maketitle

\section{Introduction} 
Let $G$ be a group. An element $g \in G$ is called an involution if $g^2=e$. It is a question of general interest whether every element of $G$ can be expressed as a product of involutions; see \cite{GHR, Li, Pa1, Pa2, RV}. It is fascinating how, even in very large groups, the product of a small number of involutions suffices to express any arbitrary element; see \cite[p. 76]{OS}. For instance, in the special linear group ${\rm SL}(n, \C)$, consisting of all $n \times n$ matrices over $\C$ with determinant $1$, every element can be written as a product of four involutions; see \cite{KN}.

A related problem is to understand the product of two involutions in a group; see \cite{bm, OS}. This problem is well-understood in the context of general linear groups; see  \cite{W, Dj, HP, Go}. It follows that for a given field $\F$, an element of the general linear group $\mathrm{GL}(n,\F)$ can be written as a product of two involutions if and only if it is conjugate to its inverse in $\mathrm{GL}(n,\F)$. However, this does not hold for the special linear group,  e.g., the unipotent Jordan block in $ \mathrm{SL}(2,  \C)$. 

An element of a group $G$ is called {\it strongly reversible} if it can be expressed as a product of two involutions. Such elements are also known as {\it strongly real} or \textit{bireflectional} in the literature. It is easy to see that an element $g \in G$ is strongly reversible if and only if $g$ is conjugate to $g^{-1}$ by an involution in $G$. When the conjugating element of $g$ is not necessarily an involution, $g$ is known as a \emph{reversible} or \emph{real} element. A conjugacy class in $G$ is called \textit{reversible conjugacy class} if it contains a reversible element, and a \textit{strongly reversible conjugacy class} if it contains a strongly reversible element.
Most investigations in the literature have focused on understanding the equivalence between the reversible and strongly reversible conjugacy classes. For an elaborate exposition of this theme, we refer to the monograph \cite{OS}. In this article, we use the notion of reversibility to classify the product of two involutions in $\mathrm{SL}(n,\C)$. Such classification is known for the finite special linear group $ \mathrm{SL}(n, q)$; see \cite{GS}.

By strong reversibility in a group $G$, we mean a classification of strongly reversible elements in $G$. Recall that an element of $\mathrm{GL}(n,\C)$ is reversible if and only if it is strongly reversible; see \cite{W, OS}. A complete list of reversible elements in $\mathrm{GL}(n,\C)$ can be found in \cite[Section 4.2]{OS}. Note that if two matrices in $ \mathrm{SL}(n,\C)$ are conjugate by an element of  $\mathrm{GL}(n,\C)$, then by a suitable scaling, we can assume that both matrices are conjugate by an element of  $ \mathrm{SL}(n,\C)$. Therefore, the classification of reversible elements in $\mathrm{SL}(n,\C)$ follows from the corresponding classification in $\mathrm{GL}(n,\C)$ and is given by the following result. We refer to \lemref{lem-Jordan-M(n, C)} for the Jordan decomposition of matrices over $\C$.

\begin{theorem}  [{\cite[Theorem 4.2]{OS}}] \label{thm-reversible-SL(n, C)}  
An element $A \in  {\rm SL }(n,\C)$  is reversible if and only if the Jordan blocks in the Jordan decomposition of $A$ can be partitioned into singletons $\{\mathrm{J}(\mu, k )\}$ or pairs $ \{ \mathrm{J}(\lambda, m),\mathrm{J}(\lambda^{-1}, m)\} $,  where $ \mu, \lambda \in \C \setminus \{0\} $ such that $ \mu \in \{-1, + 1\}$ and $\lambda \notin \{- 1, +1\}$.
\end{theorem}

However, unlike $\mathrm{GL}(n,\C)$, there are reversible elements in $\mathrm{SL}(n,\C)$ that are not strongly reversible in $\mathrm{SL}(n,\C)$. Classifying strongly reversible elements in $\mathrm{SL}(n,\C)$ is a natural problem of interest; see \cite[p. 77]{OS}.
In \cite{ST}, the authors proved that if $n \not\equiv 2 \pmod{4}$, then every reversible element in $\mathrm{SL}(n,\C)$ is strongly reversible; see \cite[Theorem 3.1.1]{ST}. Recently, in \cite{GM}, the notion of adjoint reality was introduced, and the authors used this concept to classify the strongly reversible unipotent elements in $\mathrm{SL}(n,\C)$; see \cite[Theorem 4.6]{GM}. However, these classifications do not cover all strongly reversible elements in $\mathrm{SL}(n,\C)$. We are not aware of any literature that provides a complete list of strongly reversible elements in $\mathrm{SL}(n,\C)$. The aim of this paper is to offer a complete classification for the strongly reversible elements in $\mathrm{SL}(n,\C)$.

We now state the main result of this paper which classifies products of two involutions in $\mathrm{SL}(n,\mathbb{C})$.

\begin{theorem} \label{thr-main-strong-rev-SL(n, C)}
Let $A$ be a reversible element in $\mathrm{SL}(n,\C)$. Let $s$ denote the number of Jordan blocks of the form $\{\mathrm{J}(\mu, 4k+2)\}$, $\mu \in \{-1, + 1\}$, and $t$ denote the number of pairs of the form $ \{ \mathrm{J}(\lambda, 2m+1),\mathrm{J}(\lambda^{-1}, 2m+1)\} $, $ \lambda \notin \{- 1, +1\}$, in the Jordan decomposition of $A$. Then $A$ is strongly reversible in $\mathrm{SL}(n,\C)$ if and only if at least one of the following conditions holds.
\begin{enumerate}
\item \label{cond-1-thr-strong-SL(n, C)} There is a Jordan block $\mathrm{J}(\mu, 2r+1)$, $\mu \in \{-1, + 1\}$, of odd size in the Jordan decomposition of $A$.
\item\label{cond-2-thr-strong-SL(n, C)} $s + {t} \equiv 0 \pmod 2$.
\end{enumerate}
\end{theorem}

A key ingredient of our approach used in this paper is a description of the reversers. 
For a group $ G $, the {\it centraliser and reverser} of an element $g$ of $G$ are respectively defined as 
$$\mathcal{Z}_G(g): = \{ f \in G \mid fg = gf \}, \hbox{ and \ } \mathcal{R}_G(g) := \{ h \in G \mid hgh^{-1} = g^{-1} \}.$$
The set $\mathcal{R}_G(g)$ of reversers for a reversible element $g \in G$ is a right coset of the centraliser $\mathcal{Z}_G(g)$. Thus, the \textit{reversing symmetry group} or \textit{extended centraliser} $\mathcal{E}_G(g) := \mathcal{Z}_G(g) \cup \mathcal{R}_G(g)$ is a subgroup of $G$ in which $\mathcal{Z}_G(g)$ has an index of at most $2$.
Therefore, to find the reversing symmetry group $\mathcal{E}_G(g)$ of a reversible element $g \in G$, it is sufficient to specify one reverser of $g$ that is not in the centraliser; see \cite{BR}, \cite[Section 2.1.4]{OS} for more details.

The centraliser of each element in the group $\mathrm{GL}(n,\mathbb{C})$ has been well-studied in the literature; see \cite[Proposition 3.1.2]{COV}, \cite[Theorem 9.1.1]{GLR}. It follows that to find the reversing symmetry group $\mathcal{E}_{\mathrm{GL}(n,\mathbb{C})}(A)$ of an arbitrary reversible element $A \in \mathrm{GL}(n,\mathbb{C})$, it is sufficient to find a reverser for the Jordan form of $A$; see Section \ref{sec-rev-sym-group-GL(n, C)}.  In this paper, using some combinatorial identities, we have described a reverser for certain types of Jordan forms in ${\rm GL}(n,\C)$, which are summarised in Table \ref{table:1}. We refer to \defref{def-jordan} for the notation of the Jordan block used in Table \ref{table:1}. Moreover, we require the following notation to state such an explicit description of reversers.

\begin{definition}\label{def-special-matrix-omega}
For a non-zero  $\lambda \in \C$, define $\Omega( \lambda, n) := [ x_{i,j} ]_{ n \times n}  \in   \mathrm{GL}(n,\C)$ as follows
	\begin{enumerate}
		\item $x_{i,j} = 0$ for all $1\leq i,j \leq n$ such that $j<i$,
		\item  $x_{n,n} = 1$ and $x_{i,n} = 0  $   for  all $1\leq i \leq n-1$,
		\item  \label{cond-main-special-matrix-omega} For all $ 1\leq i \leq j \leq n-1 $,  define 
		\begin{equation} \label{eq-equvi-special-matrix-omega}
			x_{i,j} =   - \lambda^{-1} x_{i+1,j} - \lambda^{-2} x_{i+1,j+1}.
		\end{equation}
\end{enumerate}
\end{definition}

{\small\begin{table}[H]	
		\centering{
			\caption{Involutory reversing symmetries for Jordan forms in $ {\GL}(n,\C)$.}
			\begin{tabular}{|p{1.5cm}|p{8cm}|p{5cm}|}
				\hline
				\vspace{0.05cm} 	\textbf{ Sr No.  } &\vspace{0.05cm}  Jordan form &  \vspace{0.05cm}  Reversing involution \\
				\hline
				\vspace{0.05cm} \hspace{.44cm}	1 & \vspace{0.005cm} 	\textit{$\mathrm{J}(\mu,  \,  n)$,  $\mu \in \{ - 1, +1 \}$}  &\vspace{0.005cm}  $\Omega(\mu,  n)$ \vspace{0.15cm} \\
				\hline
				\vspace{.33cm}	\hspace{.55cm}2 &    \vspace{0.005cm}  $ \left( \begin{array}{cc}    \mathrm{J}(\lambda, n)  &  \\  & \mathrm{J}(\lambda^{-1}, n)  \end{array}\right) $,  	$ \lambda \in \C \setminus\{-1, 0, +1\} $ \vspace{0.005cm} &  \vspace{.005cm} $ \left( \begin{array}{cc}      & \Omega(\lambda,  n) \\
					\Big(	\Omega(\lambda,  n) \Big)^{-1} &   \end{array}\right) $ \vspace{.2cm} \\ 
				\hline
			\end{tabular}
			\label{table:1}}
\end{table}}

In view of \thmref{thm-reversible-SL(n, C)} and  \lemref{lem-Jordan-M(n, C)},  we can use Table \ref{table:1} to construct a suitable reverser for each reversible element of ${\rm GL }(n,\C)$, which is also an involution in ${\rm GL }(n,\C)$. It follows that every reversible element in ${\rm GL }(n,\C)$ is strongly reversible. This constructive proof of strong reversibility in ${\rm GL }(n,\C)$ may be of independent interest.

Another key tool used in this paper is the notion of the \textit{Weyr canonical form}; see \cite{Sh1, COV}. The centraliser (and hence the reverser) of a reversible element written in Weyr canonical form is a block upper triangular matrix, making it more suitable for our purposes than the corresponding Jordan canonical form; see  \secref{subsec-intro-Weyr-form} and \remref{rem-Weyr-form-significance} for more details.

Although we restrict ourselves to the field of complex numbers, the methods used in this paper may extend to classify strongly reversible elements in the special linear group over an algebraically closed field of characteristic not equal to $2$. It is worth mentioning that in \cite{PS}, the authors proved that an element $A$ of the general linear group $\mathrm{GL}(n,\F)$ over a field $\F$  is conjugate to $-A^{-1}$ if and only if $A$ can be expressed as a product of an involution and a skew-involution (i.e., an element $\sigma$ such that $\sigma^2=-\mathrm{I}_n$); see  \cite[Theorem 5]{PS}.

Finally, consider the subgroup $\mathrm{SL}(n, \C) \ltimes \C^n$ of the affine group $\mathrm{GL}(n, \C) \ltimes \C^n$. An element $g = (A, v) \in \mathrm{GL}(n, \C) \ltimes \C^n$ acts on $\C^n$ as an affine transformation:   $g(x) = A(x) + v$. In \cite{GLM}, the authors proved that an element $g = (A, v) \in \mathrm{GL}(n, \C) \ltimes \C^n$ is reversible (respectively, strongly reversible) if and only if $A$ is reversible (respectively, strongly reversible) in $\mathrm{GL}(n, \C)$; see \cite[Theorem 1.1]{GLM}. However, the situation in $\mathrm{SL}(n, \C) \ltimes \C^n$ is more subtle than in $\mathrm{GL}(n, \C) \ltimes \C^n$. There are elements in $\mathrm{SL}(n, \C) \ltimes \C^n$ that are not  reversible in $\mathrm{SL}(n, \C) \ltimes \C^n$, even though their  linear part is  reversible in $\mathrm{SL}(n, \C)$. For example, consider the affine transformation: $x \mapsto x+1$ for all $x \in \C^n$. Reversibility in $\mathrm{SL}(n, \C) \ltimes \C^n$ is intricately related to reversibility in $\mathrm{SL}(n, \C)$; see  \cite[Lemmas 3.2--3.4]{GLM}. The results of this article, together with those in \cite{GLM}, can be used to classify the reversible and strongly reversible elements in $\mathrm{SL}(n, \C) \ltimes \C^n$. We will investigate this in a subsequent work.

\textbf{Structure of the paper.}
In \secref{sec-prelim}, we fix some notation and recall background related to the Jordan and Weyr canonical forms. In \secref{sec-rev-sym-group-GL(n, C)}, we explore the reversing symmetry groups of certain Jordan forms in $\mathrm{SL}(n,\mathbb{C})$. In \secref{sec-str-rev-Jordan-forms}, we classify the strong reversibility of certain Jordan forms by analysing the structure of the corresponding reversing symmetry groups. We also classify strongly reversible semisimple elements in \secref{sec-str-rev-Jordan-forms}. In \secref{sec-str-rev-unipotent}, we classify strongly reversible unipotent elements using the notion of the Weyr canonical form. Finally, in \secref{sec-str-rev-main-result}, we prove \thmref{thr-main-strong-rev-SL(n, C)}.

\section{Preliminaries} \label{sec-prelim}
In this section, we will fix some notations and introduce the notion of the  Jordan and Weyr canonical forms in ${\rm M }(n,\C)$, the algebra of $n \times n$ complex matrices.

\subsection{Notation for partition of $n$}\label{subsec-notation-partition-d(n)} In this section, we will recall some notation for partitioning a positive integer $n$. 
\begin{definition}[cf.~{\cite{COV}}]\label{def-partition-dual} A \textit{partition} of  a positive integer $n$ is a finite sequence $(n_1,n_2,\dots, n_r)$ of  positive integers  such that  $n_1 + n_2 + \dots + n_r =n$ and $n_1 \geq n_2 \geq \dots \geq n_r \geq  1$. Moreover, the \textit{conjugate partition} (or \textit{dual partition})  of  the partition $(n_1,n_2,\dots, n_r)$ of $n$ is the partition $(m_1,m_2,\dots, m_{n_1})$ such that $m_j = | \{i \mid n_i \geq j\}|$.
\end{definition}

For every positive integer $n$, we can represent each of its partitions using a diagram known as a Young diagram. The Young diagram of a specific partition $(n_1,n_2,\dots, n_r)$ of $n$ consists of $n$ boxes arranged into $r$ rows, where the length of the $i$-th row is $n_i$. We can obtain the Young diagram corresponding to the conjugate partition of a given partition of $n$ by flipping the Young diagram of the given partition over its main diagonal from upper left to lower right. For example, the Young diagrams corresponding to the partition $(4,4,2)$ and its conjugate partition $(3,3,2,2)$ are given as follows, respectively.
$${\small 
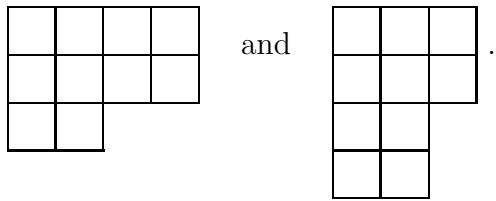
	\begin{ytableau}
		{} & {} & {} & {} \\
		{} & {} &{}  & {}\\
		{}& {}
\end{ytableau}} \quad  \hbox{and}  \quad {\small  \begin{ytableau}
		{} & {} & {}  \\
		{} & {} & {}  \\
		{}& {} \\
		{}& {}
\end{ytableau}}$$

We recall an alternative notation for partitioning a positive integer 
$n$, as introduced in \cite[Section 3.3]{GM}.

\begin{definition}\label{def-special-partition-1}
A \textit{partition }of a positive integer $n$ is a list of the form $$ {\d}(n) := [ d_1^{t_{d_1}}, \dots,   d_s^{t_{d_s}} ],$$ where $t_{d_i}, d_i \in \N, 1\leq i\leq s, $ such that $ \sum_{i=1}^{s} t_{d_i} d_i = n, t_{d_i} \geq 1 $ and  $ d_1 >  \cdots > d_s > 0$. Moreover,  for a partition $\d (n)\, =\, [ d_1^{t_{d_1}},\, \ldots ,\, d_s^{t_{d_s}} ]$ of $n$, define  ${\N}_{\d(n)}  := \{ d_i \,\mid\, 1 \,\leq\, i \,\leq\, s \}, {\E}_{\d(n)}  := {\N}_{\d(n)}  \cap (2\N)$, $ {\O}_{\d(n)}  :=  {\N}_{\d(n)}\setminus {\E}_{\d(n)}$, and $ \E_{{\d}(n)}^2 := \{ \eta \in \E_{{\d}(n)} \mid \eta \equiv 2 \pmod  {4} \}$. Furthermore, note that $ \lvert \E_{{\d}(n)}^2 \lvert =   \sum_{\eta \in \E_{{\d}(n)}^2}  t_{\eta}.$
\end{definition}

We have introduced two notations $(n_1,n_2,\dots, n_r)$ and $ {\d}(n)$ for the partition of a positive integer $n$; see   \defref{def-partition-dual} and \defref{def-special-partition-1}. 
The following lemma provides the relationship between the partition $ {\d}(n)$ and its conjugate partition $ \overline{\d}(n)$. We omit the proof as it is straightforward.
\begin{lemma}\label{lem-relation-both-partition}
Let  $ {\d}(n) = [ d_1^{t_{d_1}}, \dots,   d_s^{t_{d_s}} ]$ be a partition of a positive integer $n$,   as defined  in \defref{def-special-partition-1}. Then  the conjugate partition $ \overline{\d}(n)$ of $ {\d}(n)$ has the following form
$$ \overline{\d}(n) = \Big[ (t_{d_1}+t_{d_2}+\cdots+t_{d_s})^{d_s},  (t_{d_1}+t_{d_2}+\cdots+t_{d_{s-1}})^{d_{s-1}-d_s}, \dots,  (t_{d_1}+t_{d_2})^{d_2 -d_3},(t_{d_1})^{d_1 -d_2}\Big].$$
\end{lemma}

\subsection{Block Matrices} \label{subsec-notation-block-matrix}
We can partition a matrix $A \in  \mathrm{M}(n,\C)$ by choosing a horizontal partitioning of the rows and an independent vertical partitioning of the columns. When the same partitioning is used for both the rows and columns, we refer to the resulting partitioned matrix as a \textit{block matrix} or a \textit{blocked matrix}. For example: {\small $$A =
	\left( {\begin{array}{ccc|cc}
			2 & 0 & 0& 1 &0  \\
			0 & 2 & 0& 0 &1 \\
			0 & 0 & 2 & 0 &0\\
			\hline
			0 & 0 & 0& 2&0\\
			0& 0 & 0& 0&2 
	\end{array} } \right) = \left( {\begin{array}{c|c}
			A_{1,1}& A_{1,2}  \\
			\hline
			A_{2,1}& A_{2,2}
	\end{array} } \right) = (A_{i,j})_{1\leq i,j \leq 2} \in  \mathrm{M}(5,\C).$$}

Note that the diagonal blocks $A_{i, i}$ in the block matrix $A = (A_{i,j})_{1\leq i,j \leq m}$ are all square sub-matrices. Moreover, when specifying the block structure of matrix $A$, it is sufficient to specify only the sizes of the diagonal blocks $A_{i, i}$ since the $(i,j)$-th block $A_{i,j}$ must be a $n_i \times n_j$ matrix, where $n_i$ and $n_j$ are the sizes of the diagonal blocks $A_{i, i}$ and $A_{j,j}$, respectively. Therefore, if the diagonal blocks of $A$ have decreasing size, we can uniquely specify the whole block structure of $A$ by a partition $(n_1,n_2,\dots, n_r)$ of  $n$  such that  $n_1 + n_2 + \dots + n_r =n$ and $n_1 \geq n_2 \geq \dots \geq n_r \geq  1$.
We refer to block matrix $A$  as  \textit{block upper triangular} if $A_{i,j}=0$ for all $i > j$. If all the non-diagonal blocks of a block matrix $A$ are zero matrices,  then $A$  is called a \textit{block diagonal} matrix; see  \cite[Section 1.2]{COV} for more details.

The symbol $\mathrm{I}_r$ represents the $r \times r$ identity matrix. When $s>r$, we use the notation $\mathrm{I}_{s,r}$ to denote an $s \times r$ matrix, where the first $r$ rows form the identity matrix $\mathrm{I}_r$, and the remaining $(s-r)$ rows consist entirely of zeros. For example, $\mathrm{I}_{3,2} = \begin{psmallmatrix} 1 & 0 \\ 0&1\\ 0&0 \end{psmallmatrix}$.

\subsection{Jordan canonical form}\label{sec-Jordan} In this section, we recall the notion of the Jordan canonical form in  $\mathrm{M}(n,\C)$; see {\cite[p. 39]{COV}}, {\cite[Section 2.2]{GLR}} for more details.

\begin{definition}[{\cite[p.  52]{GLR}}]\label{def-jordan}  
A {\it Jordan block} $\mathrm{J}(\lambda,m)$ is an $m \times m$ $(m>1)$ matrix with $ \lambda \in \C$ on the diagonal entries, $1$ on all of the super-diagonal entries and $0$ elsewhere. For $m=1$, $\mathrm{J}(\lambda,1) :=(\lambda)$. We will refer to a block diagonal matrix where each diagonal block is a Jordan block as  \textit{Jordan form}. 
\end{definition}

\begin{lemma}[{Jordan  form in $ \mathrm{M}(n,\C)$, 	 \cite[Theorem 2.2.1]{GLR}}] \label{lem-Jordan-M(n, C)}
For every $A \in  \mathrm{M}(n,\C)$,  there is an invertible matrix $S \in  \mathrm{GL}(n,\C)$ such that 
	\begin{equation} \label{equ-Jordan-M(n,C)}
		SAS^{-1} =  \mathrm{J}(\lambda_1,  \,  m_1) \oplus  \cdots \oplus  \mathrm{J}(\lambda_k, \, m_k),
	\end{equation}
	where  $ \lambda_1,  \dots,  \lambda_k $ are complex numbers (not necessarily distinct).
	The form  (\ref{equ-Jordan-M(n,C)}) is uniquely determined by $A$ up to a permutation of Jordan blocks.
\end{lemma}

The Jordan structure, which lists the sizes of the diagonal blocks in the Jordan form of a matrix, is defined as follows.

\begin{definition} [{\cite[p.  39]{COV}}] \label{def-jordan-structure-partition}
Suppose $A \in \mathrm{M}(n,\C)$ is   similar to the Jordan form $\mathrm{J}(\lambda,    n_1) \oplus  \cdots \oplus  \mathrm{J}(\lambda, n_r)$ such that $n_1 + n_2 + \dots + n_r =n$ and $n_1 \geq n_2 \geq \dots \geq n_r \geq  1$. Then the partition $(n_1,n_2,\dots, n_r)$ of $n$  is called the \textit{Jordan structure}  of $A$. The Jordan structure is also known as the \textit{Segre characteristic} in the literature; see \cite{Sh2}.
\end{definition}

\subsection{Weyr canonical form} \label{subsec-intro-Weyr-form}
In this section, we recall the notion of the Weyr canonical form introduced by the Czech mathematician Eduard Weyr in 1885. The Weyr form is preferred over the Jordan form when dealing with problems concerning matrix centralisers. We refer to \cite{Sh1, COV, Sh2} for a detailed discussion on the theory of  Weyr canonical forms.

\begin{definition} [{\cite[Definition 2.1.1]{COV}}]\label{def-basic-Weyr-block}
A \textit{basic Weyr matrix} with eigenvalue $\lambda$ is a matrix $W \in  \mathrm{M}(n,\C)$ of the following form: There exists a partition $(n_1,n_2, \dots,n_r)$ of $n$ such that, when $W$ is viewed as an $r \times r$ blocked matrix $(W_{ij})_{1 \leq i,j \leq r}$, where the $(i,j)$-th block $W_{ij}$ is an $ n_i \times n_j$ matrix, the following three features hold.
	\begin{enumerate}
\item The main diagonal blocks  $W_{i,i}$  are the $ n_i \times n_i$ scalar matrices $\lambda \mathrm{I}_{n_i} $ for all $1 \leq i\leq r$.
\item The first super-diagonal blocks $W_{i,{i+1}}$ are the $ n_i \times n_{i+1}$ matrices in reduced row-echelon of rank $n_{i+1}$ (i.e., $W_{i,i+1} = \mathrm{I}_{n_{i}, n_{i+1}}$ ) for all $1 \leq i\leq r$.
\item All other blocks of $W$ are zero matrices (i.e., $W_{ij}=0$ when $j \neq i, i+1)$.
\end{enumerate}
In this case, we say that $W$ has the \textit{Weyr structure} $(n_1,n_2, \dots,n_r)$.
\end{definition}

In other words, a basic Weyr matrix is a block upper triangular matrix where the diagonal blocks are scalar matrices (i.e., scalar multiples of identity matrices). The super-diagonal blocks consist of identity matrices augmented by rows of zeros, and all the other blocks are zero matrices. An $n \times n$ scalar matrix can be viewed as a basic Weyr matrix with the trivial Weyr structure $(n)$. On the other hand, Jordan blocks are basic Weyr matrices with the Weyr structure $(1, 1, 1, \dots, 1)$.

\begin{definition}[{\cite[Definition 2.1.5]{COV}}]\label{def-general-Weyr-matrix}
Let $W \in  \mathrm{M}(n,\C)$, and let $\lambda_1, \lambda_2,\dots, \lambda_k$ be the distinct eigenvalues of $W$. We say that $W$ is in Weyr form (or is a Weyr matrix) if it is a direct sum of basic Weyr matrices, one for each distinct eigenvalue. In other words, $W$ has the following form:
$$W = W_1 \oplus W_2 \oplus \dots \oplus W_k,$$
where $W_i $ is a basic Weyr matrix with eigenvalue  $\lambda_{i}$ for all $1 \leq i\leq k$.
\end{definition}

\begin{theorem}[{\cite[Theorem 2.2.4]{COV}}]\label{thm-Weyr-canonical-form}
Up to a permutation of the basic Weyr blocks, each square matrix $A \in \mathrm{M}(n,\C)$ is similar to a unique Weyr matrix $W$. The matrix $W$ is called the Weyr (canonical) form of $A$.
\end{theorem}

In the Jordan form, a matrix with a single eigenvalue can be expressed as a direct sum of several Jordan blocks, each corresponding to that eigenvalue. In contrast, the Weyr form of a matrix with a single eigenvalue is simply the corresponding basic Weyr matrix for that eigenvalue. This highlights a significant difference between the Jordan and Weyr forms of a matrix. In this article, we will use the notation $A_W$ to denote the Weyr form corresponding to $A \in \mathrm{M}(n,\C)$. The following result recalls the centraliser of a basic Weyr matrix in $ \mathrm{M}(n,\C)$. 
\begin{proposition}[{\cite[Proposition 2.3.3]{COV}}]\label{prop-centralizer-basic-Weyr-block}
Let $W \in \mathrm{M}(n,\C)$ be an $n \times n$ basic Weyr matrix with the Weyr structure $(n_1,\dots, n_r)$, where $ r \geq 2$. Let $K$ be an $n 
	\times n$ matrix, blocked according to the partition $(n_1,\dots, n_r)$,
	and let $K_{i,j}$ denote its $(i, j)$-th block of size $n_i \times n_j$ for all $1 \leq i, j\leq r$. Then $W$ and
	$K$ commute if and only if $K$ is a block upper triangular matrix such that
	$$ K_{i,j}=  \begin{psmallmatrix}
		K_{i+1,j+1} & 
		\ast  \\
		0	&  \ast
	\end{psmallmatrix} \hbox{ for all } 1 \leq i \leq j \leq r-1.$$
Here, $K_{i,j}$ is written as a block matrix where the zero denotes the zero matrix of size $(n_i - n_{i+1}) \times n_{j+1}$. The asterisk entries $(\ast)$ indicate no restrictions on the entries in that part of the matrix. The column of asterisks disappears if $n_j = n_{j+1}$, and the row $\begin{psmallmatrix}
		0	&  \ast
	\end{psmallmatrix}$  disappears if $n_i = n_{i+1}$.
\end{proposition}

\subsection{Duality between the Jordan and Weyr Forms}
In this section, we will recall the duality between the Jordan and Weyr canonical forms.
Each partition $(n_1 , n_2 ,\dots, n_r)$ of $n$ determines a Young diagram.
The Weyr structure $(m_1,m_2,\dots,m_{n_1})$ is the conjugate partition of the Jordan structure $(n_1,n_2,\dots,n_r)$. Therefore, by transposing the Young diagram  (writing its columns as rows) of partition $(n_1, n_2,\dots, n_r)$, we get a Young diagram that corresponds to the conjugate partition $(m_1,m_2,\dots,m_{n_1})$ of $(n_1,n_2,\dots,n_r)$.  More precisely, $m_j$ is the number of $n_i$'s greater than or equal to $j$; see \defref{def-partition-dual}. Moreover, it is worth noting that if ${\d}(n)$ is a partition corresponding to the Jordan structure $(n_1,n_2,\dots,n_r)$, then the corresponding Weyr structure $(m_1,m_2,\dots,m_{n_1})$ can also be represented  by  the conjugate  partition $\overline{\d}(n)$; see \lemref{lem-relation-both-partition}.

The following result establishes the duality between Jordan and Weyr structures of complex matrices.

\begin{theorem}[{\cite[Theorem 2.4.1]{COV}}]\label{thm-relation-Weyr-Jordan-form}
The Weyr and Jordan structures of a nilpotent $n \times n$ matrix A (or a matrix with a single eigenvalue) correspond to partitions of $n$ that are conjugate (or dual) to each other. Furthermore, the Weyr and Jordan forms of a square matrix are conjugate to each other by a permutation matrix.
\end{theorem}

The following example illustrates  \propref{prop-centralizer-basic-Weyr-block} and describes the centraliser of a Weyr matrix.
\begin{example}\label{example-Weyr-form-explanation}
	Let  {\small $A = 		\left( {\begin{array}{c|c|c}
				\mathrm{J}(1,   4) & & \\
				\hline
				&	\mathrm{J}(1,   4) &\\
				\hline
				&			&	\mathrm{J}(1,   2)\\
		\end{array} } \right)$} := $\mathrm{J}(1,   4) \,  \oplus 	\, \mathrm{J}(1,   4)\,  \oplus \, 	\mathrm{J}(1,   2) $ be a unipotent Jordan form in $\mathrm{GL}(10,\C)$ with  Jordan structure $(4,4,2)$.  Then the Weyr form $
	A_{W}$ corresponding to the Jordan form $A$ has the Weyr structure $(3,3,2,2)$ and can  be given as follows
	{ \small	$$
		A_{W} = \left( {\begin{array}{c|c|c|c}
				\mathrm{I}_3 & \mathrm{I}_3 & & \\
				\hline
				&	\mathrm{I}_3 & \mathrm{I}_{3,2} &\\
				\hline
				&			&	\mathrm{I}_2& \mathrm{I}_2\\
				\hline
				&			& & \mathrm{I}_2\\
		\end{array} } \right).$$}
	Furthermore,  \propref{prop-centralizer-basic-Weyr-block} implies that a matrix $B \in \mathrm{GL}(10,\C)$ commuting with the basic  Weyr matrix  $
	A_{W} $ has the following form:
	{\small		\begin{equation} \label{eq-centralizer-example-Weyr-form}
			B =
			\left( {\begin{array}{ccc|ccc|cc|cc}
					a & b & e& h &i&l &p& q & v & w  \\
					c & d& f & j&k &m &r&s & x & y \\
					0 & 0 & g& 0 &0& n&t&u & z & \alpha \\
					\hline
					&  & & a&b &e &h& i & p & q  \\
					&  & & c &d &f &j&k  & r & s\\
					&  & & 0 &0&g &0&0 & t & u \\
					\hline
					&  & &  && &a&b & h & i \\
					&  & &  &&&c &d  & j & k \\
					\hline
					&  & &  && &&  & a & b \\
					&  & &  && &&   & c & d
			\end{array} } \right).
	\end{equation}}
Here, empty blocks represent blocks with zero matrices.
	Observe that partitions $(4,4,2)$ and $(3,3,2,2)$  representing the Weyr and Jordan structure of $A$, respectively, are conjugate (or dual) to each other; see \defref{def-partition-dual}.
	\qed
\end{example}

\section{Reversing symmetry groups in $ {\rm GL}(n,\C) $}\label{sec-rev-sym-group-GL(n, C)}
This section explores the structure of the reversing symmetry group for specific types of Jordan forms (listed in Table \ref{table:1}) in $ {\rm GL}(n,\C) $, which may be of independent interest. We will then apply these results to investigate strong reversibility in $ {\rm SL}(n,\C) $.

\begin{remark} \label{rem-equvi-special-matrix-omega}
Recall the notation  $\Omega( \lambda, n)$ introduced in  \defref{def-special-matrix-omega}. 
For more clarity, we can  write $\Omega( \lambda, n) = [ x_{i,j} ]_{ 1\leq i,  j \leq n}  \in   \mathrm{GL}(n,\C)$, where $ x_{i,j} $ are as follows
	{\small \begin{equation}\label{equation-special-matrix-omega-entries}
			x_{i,j} =\begin{cases}
				0 & \text{if $j <i $} \\
				0 & \text{if $j =n,  i  \neq n$}\\
				(-1)^{n-i} \, \lambda^{-2(n-i)}& \text{if $j=i$ }\\
				(-1)^{n-i} \,  \binom{n-i-1}{j-i}  \,  \lambda^{-2n+i+j} & \text{if $i < j, j  \neq n$ }
			\end{cases},
	\end{equation}}
	where $\lambda \in \C \setminus \{0\}$ and $\binom{n-i-1}{j-i} $ denotes the binomial coefficients. Observe that for all $ 1\leq i \leq j \leq n-1 $, we can also write Equation   \eqref{eq-equvi-special-matrix-omega} as follows
	\begin{equation} \label{eq-special-matrix-omega}
		x_{i,j} = - \lambda^{-2} x_{i+1,j+1} + \lambda^{-3} x_{i+2,j+1} - \lambda^{-4} x_{i+3,j+1} + \dots + (-1)^{(n-i)} \lambda^{-(n-i+1)}x_{n,j+1}.
	\end{equation}
\end{remark}

The following example gives relationship between $\Omega( \lambda, 4)$ and $\mathrm{J}(\lambda, 4)$ for $\lambda \neq 0$.
\begin{example}\label{examp-omega-jordan}
	For a non-zero $\lambda \in \C$,  consider $\Omega( \lambda, 4)$  as defined in \defref{def-special-matrix-omega}.
	Then
	\begin{align*}
		\Omega( \lambda, 4) \, \mathrm{J}(\lambda^{-1}, 4) & =
		\begin{psmallmatrix}
			- \lambda^{-6} & -2\lambda^{-5}   & -\lambda^{-4} & 0\\
			& \lambda^{-4} & \lambda^{-3}  &0 \\
			&  &- \lambda^{-2} &0 \\
			&  &  & 1 
		\end{psmallmatrix}  \begin{psmallmatrix}
			\lambda^{-1} & 1   &0 & 0\\
			& \lambda^{-1} & 1  &0 \\
			&  & \lambda^{-1} &1 \\
			&  &  & \lambda^{-1}
		\end{psmallmatrix}
		=\begin{psmallmatrix}
			- \lambda^{-7} & -3\lambda^{-6}   &- 3 \lambda^{-5} &- \lambda^{-4}\\
			& \lambda^{-5} & 2 \lambda^{-4}  & \lambda^{-3} \\
			&  & -\lambda^{-3} & -\lambda^{-2} \\
			&  &  & \lambda^{-1}
		\end{psmallmatrix} 
		\\&= \begin{psmallmatrix}
			\lambda^{-1} &- \lambda^{-2}   &\lambda^{-3} & -\lambda^{-4}\\
			& \lambda^{-1} &- \lambda^{-2}   &\lambda^{-3} \\
			& &  \lambda^{-1} &- \lambda^{-2}  \\
			&  &  & \lambda^{-1}
		\end{psmallmatrix}  \begin{psmallmatrix}
			- \lambda^{-6} & -2\lambda^{-5}   & -\lambda^{-4} & 0\\
			& \lambda^{-4} & \lambda^{-3}  &0 \\
			&  &- \lambda^{-2} &0 \\
			&  &  & 1 
		\end{psmallmatrix} = \Big(\mathrm{J}(\lambda, 4)\Big)^{-1} \,  \Omega( \lambda, 4).
	\end{align*}
	Furthermore,  if $\lambda \in \{-1, +1\}$, then $\lambda^{-1}= \lambda$ and $\Omega( \lambda, 4)$ is a involution.  Therefore,   $ \mathrm{J}(\lambda, 4)$ is a strongly reversible element in $ {\rm GL }(4,\C)$ for $\lambda  \in \{-1, +1\}.$
	\qed
\end{example}

The following lemma generalises Example \ref{examp-omega-jordan}.
\begin{lemma}\label{lem-conjugacy-reverser-unipotent}  Let  $\Omega( \lambda, n) \in {\rm GL }(n,\C)$ be  as defined in \defref{def-special-matrix-omega}.
	Then we have
	$$  \Omega( \lambda, n) \, \mathrm{J}(\lambda^{-1}, n) = \Big(\mathrm{J}(\lambda, n)\Big)^{-1} \,  \Omega( \lambda, n).$$
\end{lemma}

\begin{proof}
	Write $ \mathrm{J}(\lambda^{-1}, n) = [a_{i,j}]_{ n \times n} $ and $ \Big(\mathrm{J}(\lambda, n)\Big)^{-1} = [b_{i,j}]_{ n \times n} $, where 
	{\small	\begin{equation*}
			a_{i,j} =  \begin{cases}
				\lambda^{-1} & \text{if $j=i$ }\\
				1 & \text{if $j =i+1$}\\
				0 & \text{otherwise}
			\end{cases}, \hbox{ and }  b_{i,j} =  \begin{cases}
				\lambda^{-1} & \text{if $j=i$ }\\
				(-1)^k \lambda^{-(k+1)} & \text{if $j =i+k$}\\
				0 & \text{otherwise}
			\end{cases}.
	\end{equation*}}
	Let $ \Omega( \lambda, n) =[x_{i,j}]_{ n \times n} $.
	Note that   for all $1\leq i \leq n$,  we have $$\Big( \Omega( \lambda, n) \, \mathrm{J}(\lambda^{-1}, n)  \Big)_{i, i} =\Big( \Big(\mathrm{J}(\lambda, n)\Big)^{-1} \,  \Omega( \lambda, n) \Big)_{i, i} = \lambda^{-1} x_{i, i} .$$  Since matrices  under consideration are upper triangular, so it is enough to prove the following
	$$
	\Big( \Omega( \lambda, n) \, \mathrm{J}(\lambda^{-1}, n)  \Big)_{i,j} =\Big( \Big(\mathrm{J}(\lambda, n)\Big)^{-1} \,  \Omega( \lambda, n) \Big)_{i, j} = \lambda^{-1} \,  x_{i, j} + x_{i, j-1} \hbox { for all } 1 \leq i < j\leq n. 
	$$
	To see this, note that  for all $1 \leq i < j\leq n$, we have
	$$\Big( \Omega( \lambda, n) \, \mathrm{J}(\lambda^{-1}, n)  \Big)_{i,j} = \sum_{r=1}^{n} x_{i,r} \,  a_{r,j} = \sum_{r=i}^{j} x_{i,r} \,  a_{r,j}= x_{i,j-1} + x_{i,j} \lambda^{-1}=\lambda^{-1}x_{i,j} +x_{i,j-1}.$$
	Furthermore,  for all $1 \leq i < j\leq n$,  we have
	\begin{align*}
		\Big( \Big( \mathrm{J}(\lambda, n)\Big)^{-1}  \,  \Omega( \lambda, n)  \Big)_{i,j}   & = \sum_{r=1}^{n} b_{i,r}  x_{r,j} = \sum_{r=i}^{j} b_{i,r}   x_{r,j} = b_{i,i} x_{i,j} + b_{i,i+1} x_{i+1,j}+ \dots + b_{i,j} x_{j,j}
		\\&	= \lambda^{-1} x_{i,j} + (- \lambda^{-2} x_{i+1,j} +  \lambda^{-3} x_{i+2,j} +  \dots + (-1)^{(j-i)}  \lambda^{-(j-i+1)}  x_{j,j}).
	\end{align*}
	Using  Equations  \eqref{eq-equvi-special-matrix-omega} and \eqref{eq-special-matrix-omega}, we get
	$$   \Big( \Big(\mathrm{J}(\lambda, n)\Big)^{-1}  \,  \Omega( \lambda, n)  \Big)_{i,j}  = \lambda^{-1} \,  x_{i,j} + x_{i,j-1} \hbox { for all } 1 \leq i < j\leq n. 
	$$
	Hence, the proof follows.
\end{proof}

Next, we want to find a relationship between  $ \Omega( \lambda, n) $ and $ \Omega( \lambda^{-1}, n) $,  which will be used for constructing reversers that are involutions for strongly reversible elements in  ${\rm SL }(n,\C)$. For this, we will use some well-known combinatorial identities. We refer to  \cite[Section 1.2]{burton} for the basic notions related to the binomial coefficients.
Recall the following well-known binomial identities concerning binomial coefficients.
\begin{enumerate} 
	\item ~\label{eq-Pascal-rule} 
	\textit{Pascal's rule:} $  \binom{n}{k} + \binom{n}{k-1} =\binom{n+1}{k}$ for all $1 \leq k \leq n$.
	\vspace{.1cm}
	\item ~\label{eq-Newton-identity}  
	\textit{Newton's identity:} $\binom{n}{k} \binom{k}{r}= \binom{n}{r} \binom{n-r}{k-r}  $ for all $ 0\leq r \leq k \leq n$.
	\vspace{.1cm}
	\item ~\label{eq-binomial-alternate-sum}
	For $n\geq 1$,   $\sum_{k=0}^{n} (-1)^{k} \binom{n}{k}=0$.
\end{enumerate}

In the following lemma, we compute inverse of $\Omega( \lambda, n)$ in  ${\rm GL }(n,\C)$. 

\begin{lemma}\label{lem-relation-omega-inverse}
	Let  $\Omega( \lambda, n) \in {\rm GL }(n,\C)$  be as defined in \defref{def-special-matrix-omega}.
	Then we have
	$$\Big(\Omega( \lambda, n)\Big)^{-1} = \Omega( \lambda^{-1}, n).$$
\end{lemma}

\begin{proof} 
	Let $  \Omega(\lambda, n) =[x_{i,j}]_{ n \times n} $ and $ \Omega(\lambda^{-1}, n) =[y_{i,j}]_{ n \times n} $.  Then  for all $1\leq i,  j \leq n$, observe that  $x_{i,j}$ are given by Equation \eqref{equation-special-matrix-omega-entries} and 
	{\small	$$
		y_{i,j} = \begin{cases}
			0 & \text{if $j <i $}\\
			0 & \text{if $j =n,  i  \neq n$}\\
			(-1)^{n-i} \,  \lambda^{2(n-i)}& \text{if $j=i$ }\\
			(-1)^{n-i} \,  \binom{n-i-1}{j-i}  \,  \lambda^{-(-2n+i+j)}  & \text{if $ i <j, j  \neq n$. }
		\end{cases}$$}
	Here,  condition (\ref{cond-main-special-matrix-omega}) of  \defref{def-special-matrix-omega} can be checked using Pascal's rule (\ref{eq-Pascal-rule}). 
	
	Let $g = [g_{i,j}]_{n \times n} :=  \Omega(\lambda, n) \,  \Omega(\lambda^{-1}, n) $.  Then $g$ is an upper triangular matrix with diagonal entries equal to $1$, and   $g_{i,j} =   \sum_{k=1}^{n} x_{i,k} \,  y_{k,j} = \sum_{k=i}^{j} x_{i,k} \,  y_{k,j}  $ for all $1 \leq i < j \leq n$. This  implies that  for all $1 \leq i < j \leq n$, we have
	$$g_{i,j} =    \sum_{k=i}^{j}  (-1)^{n-i}  \,  \binom{n-i-1}{k-i }  \,   \lambda^{-2n+i+k}   \,  (-1)^{n-k} \,    \binom{n-k-1}{j-k }  \,  \lambda^{-(-2n+k+j)}.$$
	Therefore,  for all $1 \leq i < j \leq n$, we have
	\begin{equation}\label{eq-relation-omega-complex-inverse}
		g_{i,j} =    \lambda^{i-j} \,   \sum_{k=i}^{j}  (-1)^{(-i-k)}   \,  \binom{n-i-1}{k-i }   \,     \binom{n-k-1}{j-k }.
	\end{equation}
	By substituting $r = k-i$ in  Equation \eqref{eq-relation-omega-complex-inverse}, we get
	$$
	g_{i,j} =    \lambda^{i-j} \,   \sum_{r=0}^{j-i} (-1)^{(-2i-r)}  \binom{n-i-1}{r } \binom{(n-i-1)-r}{(j-i) -r}.
	$$
	In view of the Newton's identity (\ref{eq-Newton-identity}) and identity (\ref{eq-binomial-alternate-sum}), we get
	$$ g_{i,j}  =   \lambda^{i-j} \,  \binom{n-i-1}{j-i } \,  \sum_{r=0}^{j-i} \, (-1)^r \,  \binom{j-i}{r } =0 \hbox { for all } 1\leq i < j \leq n.$$
	Hence,   $g = \mathrm{I}_n$  in ${\rm GL }(n,\C)$.   This completes the proof.
\end{proof}

In the following result, we show that $\Omega(\mu, n)$ is an involution for $\mu \in \{-1, +1\}$.
\begin{lemma}\label{lem-involution-reverser-unipotent}
Let  $\Omega( \mu, n)  \in {\rm GL }(n,\C)$ be as defined in \defref{def-special-matrix-omega}.  If $\mu \in \{- 1, +1\}$, then  $\Omega( \mu, n)$ is an involution in ${\rm GL }(n,\C)$.
\end{lemma}
\begin{proof}
In view of the \lemref{lem-relation-omega-inverse}, we have
$$\Big(\Omega( \mu, n)\Big)^{-1} = \Omega( \mu^{-1}, n).$$
If $\mu \in \{-1, +1\}$,  then $\mu^{-1} = \mu$.  Hence, the proof follows.
\end{proof}

\begin{remark}
The	proof of Table \ref{table:1} follows from \lemref{lem-conjugacy-reverser-unipotent} and \lemref{lem-relation-omega-inverse}. \qed
\end{remark}

\subsection{Reverser set  of certain Jordan forms}\label{subsec-reverser-Jordan-block}

In this section, we will compute the reverser set $\mathcal{R}_{{\GL}(n,\C) }(A) $   for  certain Jordan forms in ${\GL}(n,\C)$. 
We introduce a notation for upper triangular \textit{Toeplitz} matrices and some definitions to formulate our results.

\begin{definition}[{\cite[p. 297]{GLR}}]\label{def-toeplitz}
For $\mathbf{x}=(x_{1},x_{2},\dots, x_{n}) \in \C^{n}$,  we define $\mathrm{Toep}_{n}(\textbf{x}) \in \mathrm{M}(n, \C)$ as
	\begin{equation}\label{eq-expression-toeplitz}
		\mathrm{Toep}_{n}(\textbf{x}):= [x_{i,j}]_{n \times n} =  {\small\begin{cases}
				0 & \text{if $i>j$ }\\
				x_{j-i+1}& \text{if $i \leq j$}\\
			\end{cases},} \hbox{ where } 1 \leq i,j \leq n.
	\end{equation}
\end{definition}
We can also write $\mathrm{Toep}_{n}(\textbf{x})$ as
\begin{equation}\label{eq-Toeplitz-matrix}
	\mathrm{Toep}_{n}(\mathbf{x})= \begin{psmallmatrix}
		x_{1} & x_{2}   & \cdots & \cdots & \cdots & \cdots & x_{n}\\
		& x_{1} & x_{2}   & \cdots  & \cdots & \cdots & x_{n-1}\\
		&  & x_{1} & x_{2}  &  \cdots  &  \cdots & x_{n-2}\\
		&  &  &\ddots   & \ddots  &&  \vdots\\
		& & & & \ddots & \ddots   & \vdots\\
		& & & &  & x_{1} & x_{2} \\
		&   &  &  &  &  & x_{1}
	\end{psmallmatrix}.
\end{equation}

\begin{definition}\label{def-general-special-matrix-omega}
For a non-zero  $\lambda \in \C$ and $\mathbf{x}=(x_{1},x_{2},\dots, x_{n}) \in \C^{n}$ such that $x_{1} \neq 0$, define upper triangular matrix
	$\Omega( \lambda,\mathbf{x}, n) := [ g_{i,j} ]_{n \times n}  \in   \mathrm{GL}(n,\C)$ as follows
	\begin{enumerate}
		\item $g_{i,j} = 0$ for all $ 1\leq i,j \leq n$ such that $i >j$,
		\item  $g_{i,n} =  x_{n-i+1}  $  for all  $1\leq i\leq n$,
		\item  $g_{i,j} =   - \lambda^{-1} g_{i+1,j} - \lambda^{-2} g_{i+1,j+1} \hbox{ for all } 1\leq i \leq j \leq n-1 $.
	\end{enumerate}
\end{definition}

Note the following example that shows the relationship between $\Omega(\lambda, 5)$ and $\Omega(\lambda, \mathbf{x}, 5)$ when $\lambda=1$.
\begin{example}\label{examp-omega-toeplitz}
	Let $\lambda= 1$ and $\mathbf{x}=(x_{1},x_{2},\dots, x_{5}) \in \C^{5}$. Note that
	$$\begin{psmallmatrix}
		x_{1} & 3x_{1} - x_{2}   & 3x_{1} -2 x_{2} + x_{3}  &  x_{1}- x_{2} +  x_{3}  - x_{4} & x_{5}\\
		& -x_{1} & -2x_{1} +x_{2}   & - x_{1} + x_{2} -x_{3}  & x_{4}\\
		&  & x_{1} & x_{1} - x_{2}  &x_{3} \\
		& &  & -x_{1} &  x_{2} \\
		&  &  &  & x_{1}
	\end{psmallmatrix} = \begin{psmallmatrix}
		x_{1} & x_{2}   &x_{3}  & x_{4} & x_{5}\\
		& x_{1} & x_{2}   &x_{3}  & x_{4}\\
		&  & x_{1} & x_{2}  &x_{3} \\
		& &  & x_{1} & x_{2} \\
		&  &  &  & x_{1}
	\end{psmallmatrix} \begin{psmallmatrix}
		1 & 3  &3  & 1 &0 \\
		& -1 & -2   & -1 & 0\\
		&  & 1 & 1  &0 \\
		& &  & -1 &0 \\
		&  &  &  & 1
	\end{psmallmatrix}.$$
	Therefore, we have 
	$\Omega( 1,\mathbf{x}, 5) =   \mathrm{Toep}_{5}(\mathbf{x})\,  \Omega( 1,5).$
	\qed\end{example}

In the following lemma, we generalise Example \ref{examp-omega-toeplitz} and establish a relationship between $\Omega(\lambda, \mathbf{x}, n)$ and the Jordan block $\mathrm{J}(\lambda, n)$.

\begin{lemma}\label{lem-relation-toep-omega}
Consider $\lambda \in \C$ and $\mathbf{x}=(x_{1},x_{2},\dots, x_{n}) \in \C^{n}$ such that $\lambda \neq 0$ and  $x_{1} \neq 0$.  Let  $\Omega( \lambda, n),\mathrm{Toep}_{n}(\mathbf{x})$,  and $\Omega( \lambda,\mathbf{x}, n)$ be as defined in \defref{def-special-matrix-omega}, \defref{def-toeplitz}, and \defref{def-general-special-matrix-omega}, respectively.
Then the following statements hold.
\begin{enumerate}
		\item $\Omega( \lambda,\mathbf{x}, n) =   \mathrm{Toep}_{n}(\mathbf{x})\,  \Omega( \lambda,n)$.
		\item $ \Omega( \lambda,\mathbf{x}, n) \, \mathrm{J}(\lambda^{-1}, n) =  \Big( \mathrm{J}(\lambda, n) \Big)^{-1} \,  \Omega( \lambda,\mathbf{x}, n)$.
\end{enumerate}
\end{lemma}

\begin{proof}
Let $ \mathrm{Toep}_{n}(\mathbf{x}) = [x_{i,j}]_{n \times n} $ and $\Omega( \lambda,n) = [y_{i,j}]_{n \times n}$. Then $ x_{i,j}$ and  $y_{i,j}$  are  given by  Equations  \eqref{eq-expression-toeplitz} and \eqref{equation-special-matrix-omega-entries}, respectively, where  $ 1 \leq i,j \leq n$.  Let   $g = [g_{i,j}]_{n \times n} = \mathrm{Toep}_{n}(\mathbf{x})\,  \Omega( \lambda,n)$. Then $g$ is an upper triangular matrix such that  for all $1 \leq i\leq j \leq n$,  we have
	\begin{equation}\label{eq-1-relation-toeplitz-reverser}
		g_{i,j} =   \sum_{k=1}^{n} x_{i,k} \,  y_{k,j} = \sum_{k=i}^{j} x_{i,k} \,  y_{k,j} =    \sum_{k=i}^{j}  x_{k-i+1}  (-1)^{n-k}  \binom{n-k-1}{j-k }  \,   \lambda^{-2n+k+j}.
	\end{equation}
	This implies that for all $1 \leq i\leq n$,  we have
	\begin{equation}\label{eq-2-relation-toeplitz-reverser}
		g_{i,i} = x_{1} (-1)^{n-i}   \lambda^{-2(n-i)},  \hbox{ and } g_{i,n} = x_{i,n} \, y_{n,n} = x_{i,n} = x_{n-i+1}.
	\end{equation}
	Note that for all $1\leq i < j \leq n-1$, we have
	\begin{align*}
		g_{i+1,j}& = \sum_{k=i+1}^{j}  x_{k-(i+1) +1} (-1)^{n-k}  \binom{n-k-1}{j-k } \, \lambda^{-2n+k+j}, \hbox{ and }
		\\	g_{i+1,j+1} &=   \Big( \sum_{k=i+1}^{j} x_{k-(i+1)+1}  (-1)^{n-k}  \binom{n-k-1}{(j+1)-k }\, \lambda^{-2n+k+j+1} \Big)
		\\ & \quad +  x_{j-i+1} (-1)^{n-j-1}  \lambda^{-2n+ 2(j+1)}.
	\end{align*}
	This implies that for all $1\leq i < j \leq n-1$, we have
	\begin{align*}
		\lambda^{-2} g_{i+1,j+1} &+ \lambda^{-1}g_{i+1,j} = x_{j-i+1}  (-1)^{n-j-1}  \,   \lambda^{-2n+ 2j}
		\\&	+ \sum_{k=i+1}^{j}   x_{k-(i+1)+1}(-1)^{n-k} \, \lambda^{-2n+k+j-1}\Bigg(  \binom{n-k-1}{(j-k) + 1 } + \binom{n-k-1}{j-k } \Bigg).
	\end{align*}
	Using Pascal's identity \eqref{eq-Pascal-rule},  for all $1\leq i < j \leq n-1$,   we have 
	\begin{align*}
		\lambda^{-2} g_{i+1,j+1} + \lambda^{-1}g_{i+1,j} &= (-1)^{n-j-1} x_{j-i+1}   \lambda^{-2n+2j}
		\\& +  \sum_{k=i+1}^{j}  (-1)^{n-k+1} x_{k-(i+1)+1}\lambda^{-2n+k+j-1} \binom{n-k}{j-k +1}.
	\end{align*}
	By substituting $r = k-1$,  for all $1\leq i < j \leq n-1$,  we have  
	\begin{align*}
		\lambda^{-2} g_{i+1,j+1} + \lambda^{-1}g_{i+1,j} &= (-1)^{n-j-1} x_{j-i+1}   \lambda^{-2n+2j}
		\\&+	\sum_{r=i}^{j-1}  (-1)^{n-r-1}  x_{r-i+1} \, \lambda^{-2n+r+j} \binom{n-r-1}{j-r}.
	\end{align*}
	This implies 
	\begin{equation}\label{eq-3-relation-toeplitz-reverser}
		\lambda^{-2} g_{i+1,j+1} + \lambda^{-1}g_{i+1,j} = 	\sum_{r=i}^{j}  (-1)^{n-r-1}  x_{r-i+1} \, \lambda^{-2n+r+j} \binom{n-r-1}{j-r},
	\end{equation}
	where $1\leq i < j \leq n-1$.
	Therefore,   from Equations \eqref{eq-1-relation-toeplitz-reverser} and  \eqref{eq-3-relation-toeplitz-reverser},  we have
	\begin{equation}\label{eq-4-relation-toeplitz-reverser}
		g_{i,j} = - ( \lambda^{-2} g_{i+1,j+1} + \lambda^{-1}g_{i+1,j})  \hbox{ for all }  1\leq i < j \leq n-1.
	\end{equation}
	Using  Equations \eqref{eq-2-relation-toeplitz-reverser}  and  \eqref{eq-4-relation-toeplitz-reverser}, we have
	$ g=  \mathrm{Toep}_{n}(\mathbf{x})\,  \Omega( \lambda,n) \,  = \, \Omega( \lambda,\mathbf{x}, n)$.
	This proves the first part of the lemma.
	
Now, recall the well-known fact that Toeplitz matrix $\mathrm{Toep}_{n}(\mathbf{x})$ commutes with the Jordan block $\mathrm{J}(\lambda, n)$ and its inverse; see \cite[Proposition 3.1.2]{COV}, {\cite[Theorem 9.1.1]{GLR}}.  Thus, using \lemref{lem-conjugacy-reverser-unipotent}, we get
	\begin{align*}
		\Big( \mathrm{Toep}_{n}(\mathbf{x})\, \Omega( \lambda, n) &\Big) \mathrm{J}(\lambda^{-1}, n) = \mathrm{Toep}_{n}(\mathbf{x})\, \Big( \Omega( \lambda,n) \, \mathrm{J}(\lambda^{-1}, n) \Big) 
	\\&	=\mathrm{Toep}_{n}(\mathbf{x})\,  \Big(\Big(\mathrm{J}(\lambda, n)\Big)^{-1} \,   \Omega( \lambda, n) \Big) =  \Big(\mathrm{J}(\lambda, n)\Big)^{-1} \,  \Big( \mathrm{Toep}_{n}(\mathbf{x})\, \Omega( \lambda, n) \Big).
	\end{align*}
	Therefore,  $ \Omega( \lambda,\mathbf{x}, n) \, \mathrm{J}(\lambda^{-1}, n) =\Big(\mathrm{J}(\lambda, n)\Big)^{-1} \,  \Omega( \lambda,\mathbf{x}, n) $.
	This completes the proof.
\end{proof}

The following result gives us the reverser set for certain reversible Jordan forms in $ {\GL}(n,\C)$. We refer to  \defref{def-special-matrix-omega} for the definition of $ \Omega(\lambda,  n)$. 

\begin{proposition} \label{prop-main-reverser-complex-Jodan-subform}
	Let $\mu, \lambda \in \C \setminus \{0\}$ such that $ \mu \in \{- 1, +1\}$ and $ \lambda \notin \{-1, +1\}$. 
	Then the following statements hold.
	\begin{enumerate} 
		\item \label{part-thm-main-reverser-complex-jodan-type-1}
		$\mathcal{R}_{{\GL}(n,\C)} (\mathrm{J}(\mu, n)) =  \Big\{   \Omega( \mu,\mathbf{x}, n) \in {\GL}(n,\C) \mid 
		\mathbf{x} \in \C^{n} \hbox {  with } x_1 \neq 0  \Big\}.$
		\item 
		\label{part-thm-main-reverser-complex-jodan-type-2}
		$\mathcal{R}_{{\GL}(2n,\C) }(\mathrm{J}(\lambda, n) \oplus \mathrm{J}(\lambda^{-1}, n))  =  $ $$  \hspace{2.7cm}\Big\{  \begin{psmallmatrix}
			& \Omega( \lambda,\mathbf{x}, n) \\
			\Omega( \lambda^{-1},\mathbf{y}, n)  &  \\
		\end{psmallmatrix} \in {\GL}(2n,\C) \mid  \mathbf{x},
		\mathbf{y} \in \C^{n} \hbox {  with } x_1, y_1 \neq 0  \Big\}.$$
	\end{enumerate}
\end{proposition}
\begin{proof} Suppose   $A \in {\GL}(n,\C)$ is a reversible element  and $ h   \in \mathcal{R}_{{\GL}(n,\C) }(A) $ is a  reverser for $A$.
Recall that the set $\mathcal{R}_{{\GL}(n,\C) }(A)$ of  reversers  of $A$ is a right coset of the centraliser  $\mathcal{Z}_{{\GL}(n,\C) }(A)$ of $A$. Therefore, if $g \in \mathcal{R}_{{\GL}(n,\C) }(A)$, i.e.,  $gAg^{-1} =A^{-1}$, then  $g = fh$ for some $f \in \mathcal{Z}_{{\GL}(n,\C) }(A))$. In other words, 
	\begin{equation}\label{eq-1-rev-complex-Jordan-subform}
		\mathcal{R}_{{\GL}(n,\C) }(A) = \mathcal{Z}_{{\GL}(n,\C) }(A)) \, h.
	\end{equation}
Now, using the above observation,  we prove the result as follows.
	
	\textit{Proof of (\ref{part-thm-main-reverser-complex-jodan-type-1}).}
	Let $ g   \in \mathcal{R}_{{\GL}(n,\C) }(\mathrm{J}(\mu, n)) $ and $h = \Omega(\mu,  n)$. Then  Lemma \ref{lem-conjugacy-reverser-unipotent}  implies  that $ h  \in \mathcal{R}_{{\GL}(n,\C) }(\mathrm{J}(\mu, n))$.   
	Using Equation \eqref{eq-1-rev-complex-Jordan-subform} and \cite[Proposition 3.1.2]{COV},  there exists a Toeplitz matrix $\mathrm{Toep}_{n}(\mathbf{x})$ such that 
	$$ g = \mathrm{Toep}_{n}(\mathbf{x}) \, h ,  \hbox{ where } \mathbf{x}=(x_{1},x_{2},\dots, x_{n}) \in \C^{n} \hbox{ and }  x_{1} \neq 0.$$
The first part of the proposition now follows from \lemref{lem-relation-toep-omega}.
	
	\textit{Proof of (\ref{part-thm-main-reverser-complex-jodan-type-2}).}
	Let $ g   \in \mathcal{R}_{{\GL}(2n,\C) }(\mathrm{J}(\lambda, n) \oplus \mathrm{J}(\lambda^{-1}, n)) $ and   $h =  \begin{psmallmatrix}
		&  \Omega( \lambda,n) \\
		\Omega( \lambda^{-1},n)   &  
	\end{psmallmatrix}$. Using Lemma \ref{lem-conjugacy-reverser-unipotent} and  \lemref{lem-relation-omega-inverse}, we have  $ h  \in \mathcal{R}_{{\GL}(2n,\C) }(\mathrm{J}(\lambda, n) \oplus \mathrm{J}(\lambda^{-1}, n))$.   
	Since $\lambda \neq \lambda^{-1}$,  Equation \eqref{eq-1-rev-complex-Jordan-subform}  and {\cite[Theorem 9.1.1,   Corollary 9.1.3]{GLR}} imply that there exist  Toeplitz matrices  $\mathrm{Toep}_{n}(\mathbf{x})$ and $\mathrm{Toep}_{n}(\mathbf{y})$  in $ {\GL}(n,\C)$ such that 
	$$ g=  \begin{psmallmatrix}
		\mathrm{Toep}_{n}(\mathbf{x}) &  \\
		&  \mathrm{Toep}_{n}(\mathbf{y})\\ 
	\end{psmallmatrix}  \begin{psmallmatrix}
		&  \Omega( \lambda,n) \\
		\Omega( \lambda^{-1},n)   &  \\ 
	\end{psmallmatrix},$$ 
	where $\mathbf{x}=(x_{1},x_{2},\dots, x_{n})$ and $\mathbf{y}=(y_{1},y_{2},\dots, y_{n})$ are in $\C^{n}$ such that both $x_{1} $ and $y_{1} $ are non-zero.   Now,  the \lemref{lem-relation-toep-omega} implies that $g$ has the following form
	$$g=  \begin{psmallmatrix}
		&   \Omega( \lambda,\mathbf{x}, n) \\
		\Omega( \lambda^{-1},\mathbf{y}, n)   &  \\ 
	\end{psmallmatrix}. $$ 
	This completes the proof.
\end{proof}

For more clarity, using \defref{def-general-special-matrix-omega}, we can rewrite the \propref{prop-main-reverser-complex-Jodan-subform}(\ref{part-thm-main-reverser-complex-jodan-type-1}) as in the following result.
\begin{corollary}
	Let $ g = [ g_{i,j} ]_{1 \leq i, j \leq n}  \in \mathcal{R}_{{\GL}(n,\C) }( \mathrm{J}(\mu, n)) $,  where $\mu  \in \{- 1, +1\}$. The entries of matrix $g$ satisfy the following conditions
	
	\begin{enumerate} 
		\item $g_{i,j} = 0$ for all $1\leq i,j \leq n$ such that $j<i$,
		\item  $g_{n,n}$ is a non-zero complex number (since $\det(g) \neq 0$),
		\item  $g_{i,n} $  is an arbitrary complex number for all  $1\leq i\leq n-1$,
		
		\item  For all $ 1\leq i \leq j \leq n-1 $, we have
		$$g_{i,j} = -  \mu^{-1} g_{i+1,j} - \mu^{-2} g_{i+1,j+1}.$$
	\end{enumerate}
\end{corollary}

\section{Strong reversibility  of certain Jordan forms in ${\SL}(n,\C)$}\label{sec-str-rev-Jordan-forms}

We will now examine the determinant of involutions in the reverser set obtained in  \propref{prop-main-reverser-complex-Jodan-subform}.

\begin{lemma} \label{lem-reverser-involution-subform}
	Let $\mu, \lambda \in \C \setminus \{0\}$ such that $ \mu \in \{- 1, +1\}$ and $ \lambda \notin \{- 1, +1\}$. Then the following statements hold.
	\begin{enumerate} 
		\item \label{part-reverser-involtuion-subform-type-1} Let $g \in \mathcal{R}_{{\GL}(n,\C)} (\mathrm{J}(\mu, n))$ be an involution. Then 
		{\small	$$\det(g) =  \begin{cases}
				+1 & \text{if  $n = 4k$}\\
				-1  & \text{if $ n =  4k +2$ }\\
				\pm 1  & \text{if  $n = 4k+1, 4k+3$}  \, (\text {i.e.,  $n$ is odd}),
			\end{cases}$$}
		where $k \in \N  \cup \{0\}$.
		
		\item  \label{part-reverser-involtuion-subform-type-2} Let $g \in \mathcal{R}_{{\GL}(2n,\C)} (\mathrm{J}(\lambda, n) \oplus \mathrm{J}(\lambda^{-1}, n))$ be an involution.  Then 
		$$\det(g) = (-1)^n = {\small \begin{cases}
				+1 & \text{if  $n$  is even}\\
				-1  & \text{if $n$ is odd.}
		\end{cases}}$$
	\end{enumerate}
\end{lemma}
\begin{proof}
	\textit{Proof of (\ref{part-reverser-involtuion-subform-type-1}).}
Using \propref{prop-main-reverser-complex-Jodan-subform}(\ref{part-thm-main-reverser-complex-jodan-type-1}), we can write $g = [g_{i,j}]_{1 \leq i, j \leq n} =  \Omega( \mu,\mathbf{x}, n)$, where $ \mathbf{x}=(x_{1},x_{2},\dots, x_{n}) \in \C^{n}$ such that $  x_{1} \neq 0$. Since $ \mu \in \{- 1, +1\}$, using  Equation \eqref{eq-2-relation-toeplitz-reverser}, we get that $ g$ is an upper triangular matrix with diagonal entries  $g_{i,i} =x_{1} (-1)^{(n-i)}  $ for all  $ 1\leq i \leq n $. This implies 
	$$ \det(g) = (x_{1})^n  \prod_{i=1}^n (-1)^{n-i} =  (x_{1})^n  (-1)^{\sum_{k=1}^n (n-i)} = (x_{1})^n (-1)^{\frac{n(n-1)}{2 }} .$$  
	Therefore,  $\det(g)$ depends only on $ x_{1}$ and $n$.
	Since $g$ is an involution with an upper triangular form, we have $(g_{n,n})^ 2= (x_1)^2 =1$. This implies  $$x_1 \in \{- 1, +1\}. $$
	The proof now follows from the equation  $ \det(g) = (x_{1})^n$  $(-1)^{\frac{n(n-1)}{2 }} $.
	
	\textit{Proof of (\ref{part-reverser-involtuion-subform-type-2}).}
	Using   \propref{prop-main-reverser-complex-Jodan-subform}(\ref{part-thm-main-reverser-complex-jodan-type-2}), we can write 
	\begin{equation}\label{equ-lem-reverser-involution-subform-2}
		g = 
		\begin{psmallmatrix}
			&   \Omega( \lambda,\mathbf{x}, n) \\
			\Omega( \lambda^{-1},\mathbf{y}, n)   &  
		\end{psmallmatrix},
	\end{equation}
	where $\mathbf{x}=(x_{1},x_{2},\dots, x_{n})$ and $\mathbf{y}=(y_{1},y_{2},\dots, y_{n})$ are in $\C^{n}$ such that both $x_{1} $ and $y_{1} $  are non-zero. Let $ \Omega( \lambda,\mathbf{x}, n) = [a_{i,j}]_{1 \leq i, j \leq n} $ and  $	\Omega( \lambda^{-1},\mathbf{y}, n) =  [b_{i,j}]_{1 \leq i, j \leq n}$.
Note that $ \Omega( \lambda,\mathbf{x}, n)$ and $	\Omega( \lambda^{-1},\mathbf{y}, n)$ are both upper triangular matrices with diagonal entries $ a_{i,i} = ( x_{1})  (-\lambda ^{-2})^{(n-i)}  $ and  $ b_{i,i} = ( y_{1}) (-\lambda ^{2})^{(n-i)}  $, respectively,  where  $ 1\leq i \leq n $. This implies
$$\det(\Omega( \lambda,\mathbf{x}, n)) = ( x_{1})^n  \Big[ \prod_{i=1}^n   (-\lambda ^{-2})^{(n-i)}  \Big] =  ( x_{1})^n  \Big[ (-\lambda^{-2})^{\sum_{i=1}^n (n-i)}  \Big].$$
	Thus, $  \det(\Omega( \lambda,\mathbf{x}, n)) = (x_{1})^n  \Big[ (-\lambda^{-2})^{\frac{n(n-1)}{2 }} \Big] $.
	Similarly, we get  $$\det(	\Omega( \lambda^{-1},\mathbf{y}, n) ) = ( y_{1})^n  \Big[ (-\lambda^{2})^{\frac{n(n-1)}{2 }}  \Big].$$
	Since $ \det(g) =  \det    \begin{psmallmatrix}
		& \Omega( \lambda,\mathbf{x}, n) \\
		\Omega( \lambda^{-1},\mathbf{y}, n) & 
	\end{psmallmatrix}$, we have   $$ \det(g) =   \det \begin{psmallmatrix}
		\Omega( \lambda,\mathbf{x}, n) &   \\
		& 	\Omega( \lambda^{-1},\mathbf{y}, n)  
	\end{psmallmatrix}   \begin{psmallmatrix}
		& \mathrm{I}_{n} \\
		\mathrm{I}_{n} &   
	\end{psmallmatrix}.$$
	Thus, $ \det(g) =  \det(\Omega( \lambda,\mathbf{x}, n)) \det(	\Omega( \lambda^{-1},\mathbf{y}, n)) (-1)^{n} $.  This implies
	$$ \det(g)  = (-1)^{n}  (- x_{1})^n ( -y_{1})^n = (-1)^{n}  ( x_{1}  y_{1})^n.$$
Since $g \in {\GL}(2n,\C)$ is an involution given in Equation \eqref{equ-lem-reverser-involution-subform-2}, where both $\Omega( \lambda,\mathbf{x}, n)$ and $\Omega( \lambda^{-1},\mathbf{y}, n)$  are upper triangular matrices in ${\GL}(n,\C)$, we have  
	$ (g^2)_{2n,2n}= x_1y_1 =1.$
	This implies	$$ \det(g)  =(-1)^{n}  ( x_{1}  y_{1})^n = (-1)^n.$$ 
	This completes the proof.
\end{proof}

\begin{remark}\label{rem-construction-rev-inv}
It is worth mentioning that the converse of \lemref{lem-reverser-involution-subform} also holds, and it follows by constructing a suitable reverser using Table \ref{table:1} and \propref{prop-main-reverser-complex-Jodan-subform}, which is also an involution. This construction is useful in proving the converse of \thmref{thr-main-strong-rev-SL(n, C)}.
\end{remark}

The next result helps us understand the strong reversibility in ${\SL}(n,\C)$.

\begin{proposition}\label{prop-str-rev-subform-SL}
	Let $\mu, \lambda \in \C \setminus \{0\}$ such that $ \mu \in \{- 1, +1\}$ and $ \lambda \notin \{-1,+1\}$.  Let $A \in {\SL}(n,\C)$ and $B \in {\SL}(2n,\C)$ denote the Jordan forms  $ \mathrm{J}(\mu, n)$ and $ \mathrm{J}(\lambda, n) \oplus \mathrm{J}(\lambda^{-1}, n)$, respectively.  Then the following statements hold.
	\begin{enumerate} 
		\item \label{part-1-prop-str-rev-subform-SL}  $A$ is reversible in ${\SL}(n,\C)$ for all $n \in \N$.
		Moreover, $A$ is strongly reversible   in ${\SL}(n,\C)$ if and only if $n \neq 4k+2$, where $k \in \N \cup \{0\}$. 
		
		\item  \label{part-2-prop-str-rev-subform-SL}  $B$ is reversible in ${\SL}(2n,\C)$ for all $n\in \N$. Moreover, 
		$B$ is strongly  reversible in ${\SL}(2n,\C)$ if and only if $n$ is even.
	\end{enumerate}
\end{proposition}
\begin{proof}
	\textit{Proof of  (\ref{part-1-prop-str-rev-subform-SL}).}
	Let	$g =   \Omega( \lambda,\mathbf{x}, n)$, where $ \mathbf{x}=(x_{1},0,\dots,0) \in \C^{n}$ such that $$  x_{1}  = {\small \begin{cases}
			-1  & \text {if $n  = 4k , 4k+1 $} \\
			+	1  & \text{if $n = 4k +3 $}\\
			-	\iota & \text{if $n = 4k +2 $}
		\end{cases},} \hbox { where } \iota^2=-1, k \in \N \cup \{0\}.$$		
	Then $ gAg^{-1} = A^{-1}$ and $g  \in {\SL}(n,\C)$.
	Thus $A$ is reversible in  ${\SL}(n,\C)$ for all $n \in \N $.
	Furthermore,   if $n \neq 4k+2$, then $g$ is also an involution in $ {\SL}(n,\C)$, where $k \in \N \cup \{0\}$.  This implies that $A$ is strongly reversible if $n \neq 4k+2$.
	
Next, consider the case when $n = 4k+2$,  where $k \in \N \cup \{0\}$.  Suppose that $A$ is strongly reversible. Then there exists $h \in {\SL}(n,\C)$ such that $ hAh^{-1} = A^{-1}$ and $h^2 = \mathrm{I}_n$.  In view of the Lemma \ref{lem-reverser-involution-subform}(\ref{part-reverser-involtuion-subform-type-1}), we have $\det(h) = -1$, which is a contradiction.  This proves the first part of the result.

	\textit{Proof of  (\ref{part-2-prop-str-rev-subform-SL}).}
	Let $	g = 
	\begin{psmallmatrix}
		&   \Omega( \lambda,\mathbf{x}, n) \\
		\Omega( \lambda^{-1},\mathbf{y}, n)   &  
	\end{psmallmatrix},$
	where $\mathbf{x}=(x_{1},0,\dots,0)$ and $  \mathbf{y}=(y_{1},0,\dots,0)$ are in $\C^{n}$ such that
	$$  x_{1} = 1,  \hbox { and } y_{1} = {\small \begin{cases}
			1  & \text{if $n$  is even} \\
			-1  & \text{if $n$ is odd }
		\end{cases}.}$$
Then $ gAg^{-1} = A^{-1}$ and $g  \in {\SL}(2n,\C)$. Thus $A$ is reversible in  ${\SL}(2n,\C)$ for all $n \in \N $. Furthermore,  if $n$ is even, then $g$ is also an involution in  ${\SL}(2n,\C)$. This implies that $A$ is strongly reversible in  ${\SL}(2n,\C)$  if $n$ is even.
	
Next, consider the case when $n$ is odd. Suppose that $A$ is strongly reversible. Then there exists $h  \in {\SL}(2n,\C)$ such that $ hAh^{-1} = A^{-1}$ and $h^2 = \mathrm{I}_{2n}$.  In view of   \lemref{lem-reverser-involution-subform}  (\ref{part-reverser-involtuion-subform-type-2}), we have $\det(h) = -1$, which is a contradiction.  This completes the proof.
\end{proof}

\subsection{Strongly reversible semisimple elements} The following lemma classifies strongly reversible semisimple (i.e., diagonalisable) elements in ${\SL}(n,\C)$. 
\begin{proposition} \label{prop-semi-strong-SL(n, C)}
A reversible semisimple element in ${\SL}(n,\C)$ is strongly reversible if and only if either $ \{\pm 1\} $ is an eigenvalue or $ n= 4k$ for some $k \in \N$. 
\end{proposition}

\begin{proof}
Let $ A \in  {\SL}(n,\C)$ be a reversible semisimple element. If $1$ or $-1$ is an eigenvalue of $A$, we can construct an involution in ${\SL}(n,\C)$ that conjugates $A$ to $A^{-1}$.   Therefore $ A $ is strongly reversible in ${\SL}(n,\C)$; see \remref{rem-construction-rev-inv}. Suppose $ 1$ and $-1 $ are not eigenvalues of $ A $. Then    \thmref{thm-reversible-SL(n, C)} implies that $ n=4m $ for some $ m\in \N $, and up to conjugacy, we can assume $ A= {\rm diag}(\lambda_1, \dots,\lambda_{2m},  \lambda_1^{-1},\dots,\lambda_{2m}^{-1})$.  Consider the involution $ g=\begin{psmallmatrix}
		&{\rm I}_{2m}\\
		{\rm I}_{2m}&\\
	\end{psmallmatrix} $ in ${\SL}(n,\C)$. Then $gAg^{-1}=A^{-1}$. Hence, $A$ is strongly reversible  in ${\SL}(n,\C)$.
	
Conversely,	suppose that $A $ is strongly reversible ${\SL}(n,\C)$ such that $ 1$ and $-1 $ are not eigenvalues of $ A $.  In view of  \thmref{thm-reversible-SL(n, C)}, up to conjugacy, we can  assume that $A$ has the form
$$ A = \Big( \lambda_1 \mathrm{I}_{m_1} \oplus {\lambda_1}^{-1} \mathrm{I}_{m_1} \Big) \oplus \Big( \lambda_2 \mathrm{I}_{m_2} \oplus {\lambda_2}^{-1} \mathrm{I}_{m_2} \Big) \oplus \dots \oplus \Big( \lambda_k \mathrm{I}_{m_k} \oplus {\lambda_k}^{-1} \mathrm{I}_{m_k} \Big),$$  where  $ \sum_{i=1}^{k}2 (m_i)  = n,  \lambda_s \neq \lambda_t $ or $\lambda_t^{-1} $ for $ s \neq t$, and $\lambda_i \notin \{- 1, +1\}$ for all $1 \leq i,s,t \leq k$.
Let  $g \in {\SL}(n,\C)$ be an  involution  such that $gAg^{-1}=A^{-1}$. 	By comparing each entry of the matrix equation $gA=A^{-1}g$ and using the conditions satisfied by each $\lambda_i$, we obtain that $g$ has the following block diagonal form
	$$g =   \begin{psmallmatrix}
		& g_1 \\
		\widetilde{g}_1 &  \\ 
	\end{psmallmatrix} \oplus \begin{psmallmatrix}
		& g_2 \\
		\widetilde{g}_2 &  \\ 
	\end{psmallmatrix} \oplus \dots \oplus \begin{psmallmatrix}
		& g_k \\
		\widetilde{g}_k &  \\ 
	\end{psmallmatrix},$$ where $g_i, \widetilde{g}_i \in  \mathrm{GL}(m_i,\C)$ for all $1 \leq i \leq k$.
	Since $g$ is an involution, we have $ \widetilde{g}_i = g_i^{-1}$ for all $1 \leq i \leq k$.
	This implies  $$\mathrm{det}(g) = \prod_{i=1}^{k} \mathrm{det} \begin{psmallmatrix}
		& g_i \\
		g_i^{-1} &  \\ 
	\end{psmallmatrix}=  \prod_{i=1}^{k} \mathrm{det} \begin{psmallmatrix}
		g_i &  \\
		& g_i^{-1}  \\ 
	\end{psmallmatrix}  \mathrm{det} \begin{psmallmatrix}
		&  \mathrm{I}_{m_i}\\
		\mathrm{I}_{m_i} &   \\ 
	\end{psmallmatrix}.$$
	Therefore, $\mathrm{det}(g) = \prod_{i=1}^{k} (-1)^{m_i} = (-1)^{\sum_{i=1}^{k} m_i} = (-1)^{\frac{n}{2}}$.  Since $ \det g =1 $,  $n$ is even.  This completes the proof.
\end{proof}

\section{Strong  reversibility  of unipotent elements  in ${\SL}(n,\C)$}  \label{sec-str-rev-unipotent}
In this section, we will investigate the strong reversibility of elements in ${\SL}(n,\C)$ with eigenvalues in $\{-1, +1\}$. First, we will focus on unipotent elements in $  \mathrm{SL}(n,\C)$. Before proceeding, we will consider an example demonstrating the complexities of studying the strong reversibility of unipotent elements in ${\SL}(n,\C)$.

\begin{example}\label{examp-str-rev-unipotent-6}
	Let $A= \begin{psmallmatrix}
		1 & 1\\
		0 & 1\\ 
	\end{psmallmatrix}  \oplus \begin{psmallmatrix}
		1 & 1\\
		0 & 1\\ 
	\end{psmallmatrix}   \oplus \begin{psmallmatrix}
		1 & 1\\
		0 & 1\\ 
	\end{psmallmatrix} $ be a unipotent element in  $  \mathrm{SL}(6,\C)$.  Then $A$ is not strongly reversible in $  \mathrm{SL}(6,\C)$. To see this, suppose that $A$ is strongly reversible. Then there exists an involution $g \in  \mathrm{SL}(6,\C)$ such that $ gAg^{-1}=A^{-1} $. Using the equation $gA= A^{-1} g$,  we get  
	\begin{equation}\label{eq-examp-rev-unipotent-6}
		g = \begin{psmallmatrix}
			x_{11} & x_{12} & x_{13} & x_{14} & x_{15} & x_{16} \\
			0 & -x_{11} &  0 & - x_{13} & 0 & -x_{15}\\ 
			x_{31} & x_{32} & x_{33} & x_{34} & x_{35} & x_{36} \\
			0 & -x_{31} &  0 & - x_{33} & 0 & -x_{35}\\ 
			x_{51} & x_{52} & x_{53} & x_{54} & x_{55} & x_{56} \\
			0 & -x_{51} &  0 & - x_{53} & 0 & -x_{55}\\ 
		\end{psmallmatrix}.
	\end{equation}
	Note that from Equation \eqref{eq-examp-rev-unipotent-6}, it is not straightforward to show that if $g$ is an involution, then $g \notin \mathrm{SL}(6,\C)$. However, 
	after suitably permuting rows and columns of matrix $g$,   we get 
	$$\widetilde{g} = S g S^{-1} =  \begin{psmallmatrix}
		x_{11} & x_{13}  & x_{15} & x_{12}  & x_{14} & x_{16} \\
		x_{31} & x_{33}  & x_{35} & x_{32}  & x_{34} & x_{36} \\
		x_{51} & x_{53}  & x_{55} & x_{52}  & x_{54} & x_{56} \\
		0 &  0 &  0 &  -x_{11} &   - x_{13}  & -x_{15}\\ 
		0 &  0 &  0 &  -x_{31} &   - x_{33}  & -x_{35}\\ 
		0 &  0 &  0 &  -x_{51} &   - x_{53}  & -x_{55}
	\end{psmallmatrix},  \hbox{ where } S \in    \mathrm{SL}(6,\C). $$
	Thus,  we can write 
	$\widetilde{g} = { \small \left(
		\begin{array}{c|c}
			P & Q\\
			\hline
			& -P
		\end{array}
		\right)}$,   where $P \in    \mathrm{GL}(3,\C)$ and $Q \in    \mathrm{M}(3,\C)$.
	Since $g$ is an involution in  $  \mathrm{SL}(6,\C)$ such that $\widetilde{g}  = Sg S^{-1}$, we have $ \widetilde{g} \in  \mathrm{SL}(6,\C) \hbox { and }   {(\widetilde{g})}^2 =  \mathrm{I}_6$.
	Using  ${(\widetilde{g})}^2 =  \mathrm{I}_6$, i.e.,   $P^2 =  \mathrm{I}_3$, we have
	$ \mathrm{det} (\widetilde{g})  = \mathrm{det}(P)\mathrm{det}(-P)=\mathrm{det}(P^2) (-1)^3 = (-1)^3.$
	This is a contradiction. Hence, $A$ is not  strongly reversible in $  \mathrm{SL}(6,\C)$. 
	\qed
\end{example}

\begin{remark}\label{rem-Weyr-form-significance}
Recall that every unipotent matrix in $  \mathrm{SL}(n,\C)$ is reversible, and we know the structure of the set of corresponding reversers; see \thmref{thm-reversible-SL(n, C)} and \secref{sec-rev-sym-group-GL(n, C)}. We are interested in finding the necessary conditions for strong reversibility of a unipotent matrix. In Example \ref{examp-str-rev-unipotent-6}, we have transformed the reverser $g$ into a block upper triangular form by appropriately permuting its rows and columns. This step is crucial in proving that the unipotent matrix considered in Example \ref{examp-str-rev-unipotent-6} is not strongly reversible. For a unipotent matrix $A$ in Jordan form with Jordan blocks of unequal sizes or a large number of diagonal Jordan blocks, checking the existence of an involution in  $\mathcal{R}_{{\GL}(n,\C) }(A) \cap {\SL}(n,\C)$  becomes challenging. Therefore, using the Jordan canonical form of a matrix to study strongly reversible elements in ${\SL}(n,\C) $ is not an efficient approach. Instead, we can use the Weyr canonical form, which has a more suitable centraliser (and reverser) than the Jordan canonical form; see \propref{prop-centralizer-basic-Weyr-block}. In particular, if a reversible element of ${\SL}(n,\C) $ is in the Weyr canonical form, every reverser has a block upper triangular form. 
\qed
\end{remark}

Now, using the notion of the Weyr form, we investigate the strong reversibility of the unipotent Jordan form, where all the Jordan blocks are of the same size. The following result generalises Example \ref{examp-str-rev-unipotent-6}.

\begin{lemma}\label{lem-general-uni-examp-non-str-rev}
Let $A=  \bigoplus_{ i=1}^{k}   \mathrm{J}(1, 2m)$ be a unipotent Jordan form in $\mathrm{SL}(2mk,\C) $.  Then the following statements hold.
	\begin{enumerate}
\item[(\textit{i})] If $gAg^{-1} = A^{-1}$ and $g^2 =  \mathrm{I}_{2mk}$, then  $\mathrm{det}(g)= (-1)^{mk}$.
\item[(\textit{ii})] If there are an odd number of Jordan blocks of size $2 \pmod 4$ (i.e., $m$ and $ k $ both are odd), then $A$ can not be strongly reversible in  $\mathrm{SL}(2mk,\C)$.
\end{enumerate}
\end{lemma}
\begin{proof}
Let $A_{W}  \in \mathrm{SL}(2mk,\C)$ be the Weyr form corresponding to the Jordan form $A$; see \defref{def-basic-Weyr-block}. Then using \thmref{thm-relation-Weyr-Jordan-form}, we have
$A_{W} = \tau  A \tau^{-1} \hbox{ for some }  \tau \in \mathrm{SL}(2mk,\C).$ Moreover, $ A_{W} $  has the  Weyr structure $( \underset{2m\text{-times}}{\underbrace{{k,k,\dots,k} }})$
	and can be  written as follows
	\begin{equation}\label{eq-unipotent-Weyr-block}
		A_{W} = \begin{psmallmatrix}
			\mathrm{I}_{k} &\mathrm{I}_{k}   &  &  &    \\
			& \mathrm{I}_{k} & \mathrm{I}_{k}   &  &     \\
			&  &\ddots   & \ddots  &  \\
			& &  & \mathrm{I}_{k} & \mathrm{I}_{k} \\
			&  &  &  &\mathrm{I}_{k}
		\end{psmallmatrix}.
	\end{equation}
	Now, define $\Omega(\mathrm{I}_{k}, 2m) := [ X_{i,j} ]_{1 \leq i,j \leq 2m}  \in   \mathrm{GL}(2mk,\C)$  such that
	\begin{enumerate} 
		\item $X_{i,j} = \mathrm{O}_{k}$ for all $1\leq i,j \leq m$ such that $j<i$, where $\mathrm{O}_{k}$  denotes the $k \times k$ zero matrix.
		\item $X_{2m,2m} = \mathrm{I}_{k}$ and  $X_{i,2m} = \mathrm{O}_{k}  $  for all  $1\leq i \leq 2m-1$,
		\item  $X_{i,j} =    - X_{i+1,j} -  X_{i+1,j+1} \hbox{ for all }  1\leq i  \leq j \leq 2m-1$.	
	\end{enumerate}
	By using a similar argument as in \lemref{lem-conjugacy-reverser-unipotent}, we have
	\begin{equation}\label{eq-rev-uniptent-Weyr}
		\Omega(\mathrm{I}_{k}, 2m) \,A_{W} = A_W^{-1}  \, \Omega(\mathrm{I}_{k}, 2m).
	\end{equation}
	Furthermore, if $B \in \mathrm{GL}(2mk, \C)$  commutes with $A_{W}$, then  \propref{prop-centralizer-basic-Weyr-block} implies that   
	\begin{equation}\label{eq-centralizer-uniptent-Weyr}
		B=  \mathrm{Toep}_{2m}(\mathbf{K}):= \begin{psmallmatrix}
			{K}_{1,1} &{K}_{1,2}   & \cdots &  \cdots  & \cdots  & {K}_{1,2m} \\
			& {K}_{1,1} & {K}_{1,2}   & \cdots  & \cdots  &  {K}_{1,2m-1} \\
			&  &\ddots   & \ddots  && \vdots \\
			& & & \ddots & \ddots   & \vdots \\
			& & &  & {K}_{1,1} & {K}_{1,2} \\
			&  &  &  &  &{K}_{1,1}
		\end{psmallmatrix},
	\end{equation}
	where  $\mathbf{K} = (K_{1,1}, K_{1,2}, \dots,K_{1,2m})$  is a $2m$-tuple of matrices such that ${K}_{1,1} \in \mathrm{GL}(k,\C)$ and ${K}_{1,j} \in \mathrm{M}(k,\C)$ for all $2 \leq j \leq 2m$.  
	
Let $ h = [Y_{i,j}]=\tau g \tau^{-1 }  \in \mathrm{GL}(2mk,\C)$, where  $Y_{i,j} \in \mathrm{M}(k,\C)$ for all $1 \leq i, j \leq 2m$.  Since $g \in \mathrm{GL}(2mk,\C)$ such that $gAg^{-1} = A^{-1}$ and $g^2 =  \mathrm{I}_{2mk}$, we have 
$$ h A_{W}  h^{-1} =   A_{W} ^{-1} \hbox{ and } h^2 = \mathrm{I}_{2mk}.$$   
Note that the set of reversers of $A_{W}$ is a right coset of the centraliser of $A_{W}$. Therefore,  using Equations \eqref{eq-rev-uniptent-Weyr} and \eqref{eq-centralizer-uniptent-Weyr},  we have 
	\begin{equation*}\label{eq-reversing-symmetry-unipotent-Weyr-block}
		h =  \mathrm{Toep}_{2m}(\mathbf{K}) \, \Omega(\mathrm{I}_{k}, 2m)  \hbox{ and } h^2 = \mathrm{I}_{2mk}.
	\end{equation*}
	where $\mathrm{det}(K_{1,1}) \neq 0$.
	It follows that $h$ is a block upper-triangular matrix with diagonal blocks $Y_{i,i} = (-1)^{(2m-i)} K_{1,1}$ for all $1 \leq  i \leq 2m$ such that $(K_{1,1})^2= \mathrm{I}_k$. This implies
	$$ \mathrm{det}(h)= \prod_{i=1}^{2m} \mathrm{det}\Big((-1)^{(2m-i)}K_{1,1} \Big)  = ((-1)^{k})^{m } \, \Big(\mathrm{det}\Big((K_{1,1})^2 \Big)\Big)^m = (-1)^{km}.$$
	Since $ h =\tau g \tau^{-1 } $, we have $\mathrm{det}(g)= (-1)^{mk}$. Therefore, if $m $ and $k$ are odd, then $\mathrm{det}(g)= -1$. Hence, $A$ is not strongly reversible in   $\mathrm{SL}(2mk,\C)$ if $m $ and $k$ are odd. This completes the proof.
\end{proof}

Next, we consider an arbitrary unipotent Jordan form in  $\mathrm{SL}(n,\C)$.
Before that, note the following example, which gives an idea of the proof of the general result (cf. \propref{prop-unipotent-strong-rev-SL(n, C)}).
\begin{example}\label{example-non-st-rev-non-homo-unipotent}
	Let $A$ be a unipotent element in $\mathrm{SL}(10,\C)$  with Jordan structure $(4,4,2)$ as given in Example \ref{example-Weyr-form-explanation}. Then $A$ is not be strongly reversible in $\mathrm{SL}(10,\C)$. To see this, recall that $A$ has the Weyr structure $(3,3,2,2)$. Therefore, the Weyr form $A_{W}$  of $A$  and its inverse $A^{-1}_{W} $ can be given as 
	$$	{\small	
		\left( {\begin{array}{c|c|c|c}
				\mathrm{I}_3 & \mathrm{I}_3 &  &  \\
				\hline
				& \mathrm{I}_3 & \mathrm{I}_{3,2}&  \\
				
				\hline
				&  & \mathrm{I}_2& \mathrm{I}_2   \\
				
				\hline
				&  & & \mathrm{I}_2 
		\end{array} } \right), \hbox{ and } 
		A^{-1}_{W} = \left( {\begin{array}{c|c|c|c}
				\mathrm{I}_3 & -\mathrm{I}_3 & \mathrm{I}_{3,2} & - \mathrm{I}_{3,2} \\
				\hline
				& \mathrm{I}_3 & - \mathrm{I}_{3,2}&  \mathrm{I}_{3,2} \\
				\hline
				& & \mathrm{I}_2&- \mathrm{I}_2   \\
				
				\hline
				&  & & \mathrm{I}_2 
		\end{array} } \right).
	}$$
	In view of Equation \eqref{eq-centralizer-example-Weyr-form}, any element $f \in \mathrm{GL}(10,\C)$ that satisfies $fA_{W} = A_{W} f$ has the following block upper triangular  form
	$$ f= {\small \left( {\begin{array}{c|c|c|c}
				\begin{psmallmatrix}
					P & \ast \\
					& Q
				\end{psmallmatrix} & \ast & \ast & \ast \\
				\hline
				& \begin{psmallmatrix}
					P & \ast \\
					& Q
				\end{psmallmatrix} &\ast&  \ast \\
				\hline
				& & 	P & \ast  \\
				\hline
				&  &  & 	P
		\end{array} } \right),}
	\hbox{ where } 	P  = \begin{psmallmatrix}
		a & b\\
		c & d
	\end{psmallmatrix} \in \mathrm{GL}(2,\C), \, Q = \begin{pmatrix}
		g
	\end{pmatrix}\in \mathrm{GL}(1,\C).$$
	Let
	$
	h= 	{\small \left( {\begin{array}{c|c|c|c}
				-	\mathrm{I}_3 & - 2\mathrm{I}_3 & -\mathrm{I}_{3,2} &  \\
				\hline
				& \mathrm{I}_3 & \mathrm{I}_{3,2}& \\
				
				\hline
				&  & -\mathrm{I}_2&  \\
				
				\hline
				&  & & \mathrm{I}_2 
		\end{array} } \right)}$. Then
	$ h A_{W}  = A^{-1}_{W} h= {\small	\left( {\begin{array}{c|c|c|c}
				-	\mathrm{I}_3 & - 3\mathrm{I}_3 & -3\mathrm{I}_{3,2} & - \mathrm{I}_{3,2} \\
				\hline
				& \mathrm{I}_3 &2 \mathrm{I}_{3,2}&  \mathrm{I}_{3,2} \\
				
				\hline
				& & -\mathrm{I}_2&  -\mathrm{I}_{2}   \\
				
				\hline
				&&  & \mathrm{I}_2 
		\end{array} } \right)}.$
	Therefore, $h A_{W} h^{-1} =A^{-1}_{W}$. 	Since the  set of  reversers of $A_{W}$ is a right coset of the centraliser of $A_{W}$, every  reverser $\tau$  of $A_{W}$  has the following form
	$$\tau = fh=  {\small\left( {\begin{array}{c|c|c|c}
				\begin{psmallmatrix}
					-	P & \ast \\
					&- Q
				\end{psmallmatrix} & \ast & \ast & \ast \\
				\hline
				& \begin{psmallmatrix}
					P& \ast \\
					& Q
				\end{psmallmatrix} &\ast&  \ast \\
				\hline
				& & -P & \ast  \\
				\hline
				&  &  & P
		\end{array} } \right)}
	.$$
	This implies that  $\mathrm{det}(\tau) =	((-1)^2)^2 \mathrm{det}(P^4) (-1)^{1} \mathrm{det}(Q^2)=(-1)  \mathrm{det}(P^4)\mathrm{det}(Q^2).$
	Moreover, if $\tau$ is an involution, then both $ P$ and $Q$ are also involutions, and thus $\mathrm{det}(P^4) = \mathrm{det}(Q^2)=1$. Therefore, if  $\tau  \in \mathrm{GL}(10,\C)$ such that $\tau A_{W} \tau ^{-1} = A_{W} ^{-1}$ and $\tau^2 = \mathrm{I}_{10}$, then  $\mathrm{det}(\tau) =-1$. Hence, $A$ is not strongly reversible in $\mathrm{SL}(10,\C)$.
	\qed
\end{example}

Now, we will generalise Example \ref{example-non-st-rev-non-homo-unipotent}.  The following result follows from the proof of \cite[Theorem 4.6]{GM}. However, we will provide an alternative proof using the notion of the Weyr form. We refer to Section \ref{subsec-notation-partition-d(n)} for  notation of the partition used in the following result.
\begin{proposition}\label{prop-det-reverser-inv-unipotent-SL(n, C)}
Let $A \in \mathrm{SL}(n,\C)$ be a unipotent element such that the Jordan decomposition  of $A$ is represented by the  partition $ {\d}(n) = [ d_1^{t_{d_1}}, \dots,   d_s^{t_{d_s}} ] $, where $d_{k}$ is even for all $1 \leq k \leq s$. If $g$ is  an involution in $\mathrm{GL}(n,\C)$  such that  $gAg^{-1} =A^{-1}$, then $\mathrm{det}(g) = (-1)^{\lvert \E_{{\d}(n)}^2 \rvert}.$
\end{proposition}
\begin{proof}
	Let $A_{W}$ denote the Weyr form of $A$. 
	Using \lemref{lem-relation-both-partition}, the partition $\overline{{\d}}(n) $ representing the Weyr form $A_{W}$ is given by
	$$ \overline{\d}(n) = \Big[ (t_{d_1}+t_{d_2}+\cdots+t_{d_s})^{d_s},  (t_{d_1}+t_{d_2}+\cdots+t_{d_{s-1}})^{d_{s-1}-d_s}, \dots,  (t_{d_1}+t_{d_2})^{d_2 -d_3},(t_{d_1})^{d_1 -d_2}\Big].$$
	Therefore,  $A_{W}$ is a block matrix with $(d_1)^2$ many blocks. Moreover,  for all $1\leq i \leq d_1$, the size $n_i$ of the $i$-th diagonal block of  $A_{W}$ is given by
	$$ n_i =  {\small \begin{cases}
			t_{d_1}+t_{d_2}+\cdots+t_{d_s}  & \text{if $1 \leq i \leq d_s$ }\\
			t_{d_1}+t_{d_2}+\cdots+t_{d_{s-r-1} }& \text{if ${d_{s-r}} +1  \leq i \leq {d_{s-r-1}}$},  \text{where $0 \leq r \leq s-2$}
		\end{cases}.}$$ 
	This implies that the $(i,j)$-th block of $A_{W}$  has size $ n_i \times  n_j$, where $1\leq i,j \leq d_1$; see Section \ref{subsec-notation-block-matrix}. Furthermore, the $(i,j)$-th blocks of the Weyr form $A_{W}$ and its inverse $A_{W} ^{-1}$  are given by
	$$ {\small (A_{W})_{i,j} =  \begin{cases}
			\mathrm{I}_{n_i} & \text{if $i=j$ }\\
			\mathrm{I}_{n_i, n_{i +1}} & \text{if $i+1=j$}\\
			\mathrm{O}_{n_i \times n_{j}}  & \text{otherwise}
		\end{cases},} \text{ and }   {\small (A_{W} ^{-1})_{i,j} =  \begin{cases}
			\mathrm{I}_{n_i} & \text{if $i=j$ }\\
			(-1)^{(j-i)} \,  \mathrm{I}_{n_i, n_{j}} & \text{if $i<j$}\\
			\mathrm{O}_{n_i \times n_{j}}  & \text{otherwise}
		\end{cases},}$$
	where $ 1 \leq i,j \leq d_1 $. 
	Consider $\Omega_{W} := [ X_{i,j} ]_{1 \leq i,j \leq d_1}  \in   \mathrm{GL}(n,\C)$  such that
	{\small	\begin{equation}
			X_{i,j} =\begin{cases}
				\mathrm{O}_{n_i \times n_{j}}  & \text{if $j<i $}\\
				\mathrm{O}_{n_i \times n_{j}}  & \text{if $j =d_1,  i  \neq d_1$}\\
				(-1)^{d_1-i} \, \mathrm{I}_{n_i}& \text{if $j=i$ }\\
				(-1)^{d_1-i} \,  \binom{d_1-i-1}{j-i}  \,  \mathrm{I}_{n_i, n_{j}}  & \text{if $ i<j, j  \neq d_1$ }
			\end{cases},
	\end{equation}}
	where  $\binom{d_1-i-1}{j-i} $ denotes the binomial coefficients. Then by using a similar argument as in \lemref{lem-conjugacy-reverser-unipotent}, we have 
	$$\Omega_{W} A_{W} \Omega_{W}^{-1}=A_{W}^{-1}.$$
Let $f=[P_{i,j}]_{1 \leq i,j \leq d_1} \in  \mathrm{GL}(n,\C)$ be an $n \times n$ matrix  commuting with  Weyr form $A_{W}$ such both $f$ and $A$ has the same  block structure given by the partition $\overline{{\d}}(n) $. Then using \propref{prop-centralizer-basic-Weyr-block}, we can conclude that $f$ is an upper triangular block matrix, and the $i$-th diagonal block $P_{i,i}$ of $f$ has the following form
	\begin{enumerate}
		\item if  $1 \leq i \leq d_s$, then $ P_{i,i}=  \begin{psmallmatrix}
			{P}_{1} &\ast  & \ast &  \ast  & \cdots  &  \ast \\
			& {P}_{2} &  \ast   & \cdots  & \cdots  &   \ast \\
			&  &\ddots   & \ddots  && \vdots \\
			& & & \ddots & \ddots   & \vdots \\
			& & &  & {P}_{{s-1}} &  \ast \\
			&  &  &  &  &{P}_{s}
		\end{psmallmatrix}$,
		
		\item if  ${d_{s-r}} +1  \leq i \leq {d_{s-r-1}}$, then $$ P_{i,i}=  \begin{psmallmatrix}
			{P}_{1} &\ast  & \ast &  \ast  & \cdots  &  \ast \\
			& {P}_{2} &  \ast   & \cdots  & \cdots  &   \ast \\
			&  &\ddots   & \ddots  && \vdots \\
			& & & \ddots & \ddots   & \vdots \\
			& & &  & {P}_{{s-r-2}} &  \ast \\
			&  &  &  &  &{P}_{s-r-1}
		\end{psmallmatrix},$$
		
		\item  if  $d_2 +1  \leq i \leq d_1$, then $P_{i,i} = P_1$,
	\end{enumerate}
	where  $0 \leq r \leq s-3$ and $P_k \in \mathrm{GL}(t_{d_k}, \C)$ for all $ 1 \leq k \leq s$.
	It is worth noting that there are $d_k$ many times matrices $P_k$ are repeated in the diagonal blocks of matrix $f$, where $ 1 \leq k \leq s$. 
	
Since $A_{W}$ is the Weyr form of $A$, there exists a $\tau \in \mathrm{GL}(n,\C)$ such that $A_{W} = \tau A \tau^{-1}$.	Consider $h = \tau g \tau^{-1} = [Y_{i,j}]_{1 \leq i,j \leq d_1} \in \mathrm{GL}(n,\C)$. Since $g$ is  an involution in $\mathrm{GL}(n,\C)$  such that  $gAg^{-1} =A^{-1}$, it follows that $h$ is an involution in $\mathrm{GL}(n,\C)$ such that 
	\begin{equation}\label{eq-relation-reverser-Weyr-Jordan-unipotent}
		h A_{W} h^{-1} =A_{W}^{-1}, \hbox{ and } \mathrm{det}(h) =\mathrm{det}(g).
	\end{equation}
	Note that the set of reversers of $A_{W}$ is a right coset of the centraliser of $A_{W}$. Therefore,
	$$h =  f  \,\Omega_{W}.$$
	This implies  $h$ is an upper triangular block matrix such that diagonal blocks of $h$ are given by  $$Y_{i, i} = (-1)^{d_1-i}P_{i, i}, \hbox{ where }  1 \leq i \leq d_1.$$
	Therefore, $\mathrm{det}(h) =	\prod_{i=1}^{d_1} \mathrm{det}( (-1)^{d_1-i}P_{i,i})$. 	 Since  $d_k$ is even for all $1 \leq k \leq s$, we have $$ \mathrm{det}(h) = ((-1)^{t_{d_1}})^{\frac{d_1}{2}}  \mathrm{det}((P_{1})^{2})^{\frac{d_1}{2}}   ((-1)^{t_{d_2}})^{\frac{d_2}{2}} \mathrm{det}((P_{2})^{2})^{\frac{d_2}{2}}  \cdots ((-1)^{t_{d_s}})^{\frac{d_s}{2}}\mathrm{det}((P_{s})^{2})^{\frac{d_s}{2}}.$$
	Since $h$ is an involution,   $P_{k}$ is an involution for all $1 \leq k\leq s$. This implies 
	$$\mathrm{det}(h) =  ((-1)^{t_{d_1}})^{\frac{d_1}{2}}((-1)^{t_{d_2}})^{\frac{d_2}{2}} \cdots ((-1)^{t_{d_s}})^{\frac{d_s}{2}}.$$
	Observe that if $ d_k = \,0 \pmod 4$ for some $1 \leq k\leq s$, then  $((-1)^{t_{d_k}})^{\frac{d_k}{2}} = 1$. Thus,  we have
	$$\mathrm{det}(h) = \prod_{d_k = 2\pmod4} ((-1)^{t_{d_k}})^{\frac{d_k}{2}} =  \prod_{d_k =2 \pmod 4} (-1)^{t_{d_k}} =(-1)^{\sum_{d_k \in \E_{{\d}(n)}^2}  t_{d_k}} = (-1)^{\lvert \E_{{\d}(n)}^2 \rvert}, $$
	where $\E_{{\d}(n)}^2 = \{ d_k \mid d_k \equiv 2  \pmod 4 \}$. 
	The proof now follows from Equation \eqref{eq-relation-reverser-Weyr-Jordan-unipotent}.
\end{proof}

Finally, we classify the strongly reversible unipotent elements in $\mathrm{SL}(n,\C)$. This result is also proved in \cite{GM} using an infinitesimal version of the notion of the classical reversibility or reality, known as \textit{adjoint reality}.

\begin{proposition}[{\cite[Theorem 4.6]{GM}}] \label{prop-unipotent-strong-rev-SL(n, C)}
Let $A \in \mathrm{SL}(n,\C) $ be a unipotent element. Then $A$ is strongly reversible if and only if at least one of the following conditions holds.
	\begin{enumerate}
		\item \label{cond-lem-uni-str-rev-1}  There is a Jordan block $\mathrm{J}(1, 2r+1)$ of odd size in the Jordan decomposition of $A$.

		\item \label{cond-lem-uni-str-rev-2}  The total number of Jordan blocks of the form $\mathrm{J}(1, 4k+2)$  in the Jordan decomposition of $A$ is even.
	\end{enumerate}
\end{proposition}

\begin{proof} Up to conjugacy, we can assume that $A$ is in  Jordan form. Let $ {\d}(n) = [ d_1^{t_{d_1}}, \dots,   d_s^{t_{d_s}} ]$ be the  partition representing the Jordan form of $A$.  Note that if $A$ has a Jordan block of odd size, then using  \remref{rem-construction-rev-inv} and \propref{prop-str-rev-subform-SL}, we can construct a suitable involution $g $ in $ \mathrm{SL}(n,\C) $ such that $ gAg^{-1} =A^{-1}$. 	
 Let  $A$ be a strongly reversible element in $ \mathrm{SL}(n,\C) $. Then there exists an involution $g \in \mathrm{SL}(n,\C) $ such that $ gAg^{-1} =A^{-1}$.  Suppose that there does not exist any Jordan block of odd size in the Jordan decomposition of $A$, otherwise we are done.  Then  $d_{k}$ is even for all $1 \leq k \leq s$.  
Using \propref{prop-det-reverser-inv-unipotent-SL(n, C)}, we have
$$\mathrm{det}(g) = (-1)^{\lvert \E_{{\d}(n)}^2 \rvert} =1.$$
This implies that  $\lvert \E_{{\d}(n)}^2 \rvert  $ is even, where $\E_{{\d}(n)}^2 = \{ d_k \mid d_k \equiv 2  \pmod 4 \}$ is as defined in \defref{def-special-partition-1}. Thus, the forward direction of the result is proven.
	
Conversely, let at least one  of the conditions $(\ref{cond-lem-uni-str-rev-1})$ or $(\ref{cond-lem-uni-str-rev-2})$ of \propref{prop-unipotent-strong-rev-SL(n, C)} holds. Then  using  \lemref{lem-Jordan-M(n, C)}, \remref{rem-construction-rev-inv}, and \propref{prop-str-rev-subform-SL}, we can construct a suitable involution $g $ in $ \mathrm{SL}(n,\C) $ such that $ gAg^{-1} =A^{-1}$. Therefore, $A$ is strongly reversible in $ \mathrm{SL}(n,\C) $. This completes the proof.
\end{proof} 

The study of strong reversibility of an element in $\mathrm{SL}(n, \C)$ with $-1$ as its only eigenvalue is analogous to the unipotent case. Similar to \propref{prop-det-reverser-inv-unipotent-SL(n, C)}, we have the following result, which will be used in proving \thmref{thr-main-strong-rev-SL(n, C)}.

\begin{proposition} \label{prop-det-reverser-inv-SL(n, C)-eigenvalue-minus-1}
Let $A \in \mathrm{SL}(n,\C)$ be an element with $-1$ as its only eigenvalue such that the Jordan decomposition  of $A$ is represented by the  partition $ {\d}(n) = [ d_1^{t_{d_1}}, \dots,   d_s^{t_{d_s}} ] $, where $d_{k}$ is even for all $1 \leq k \leq s$. If $g$ is  an involution in $\mathrm{GL}(n,\C)$  such that  $gAg^{-1} =A^{-1}$, then $\mathrm{det}(g) =
	(-1)^{\lvert \E_{{\d}(n)}^2 \rvert}.$
\end{proposition}
\begin{proof}
The proof follows using a similar line of arguments as in \propref{prop-det-reverser-inv-unipotent-SL(n, C)}.
\end{proof}

The next result classify strongly reversible elements in $\mathrm{SL}(n, \C)$ with $-1$ as their only eigenvalue.

\begin{proposition} \label{prop-strong-rev-SL(n, C)-eigenvalue-minus-1}
Let $A$ be an element of $\mathrm{SL}(n,\C)$ with $-1$ as its only eigenvalue. Then $A$ is strongly reversible if and only if at least one of the following conditions holds.
\begin{enumerate}
\item  There is a Jordan block $\mathrm{J}(-1, 2r+1)$ of odd size in the Jordan decomposition of $A$.
\item   The total number of Jordan blocks of the form $\mathrm{J}(-1, 4k+2)$  in the Jordan decomposition of $A$ is even.
	\end{enumerate}
\end{proposition}

\begin{proof}
	The proof follows using a similar line of arguments as in \propref{prop-unipotent-strong-rev-SL(n, C)}.
\end{proof}

\section{Strong reversibility  in ${\SL}(n,\C)$ } \label{sec-str-rev-main-result}
In this section, we will investigate the strong reversibility of reversible elements of $\mathrm{SL}(n,\C)$ having eigenvalues $\lambda$ and $ \lambda^{-1}$, where $\lambda \neq \pm1$, and prove our main result \thmref{thr-main-strong-rev-SL(n, C)}. 
\begin{proposition}\label{prop-det-rev-inv-non-unipotent-Weyr}
	Let $\lambda \neq \pm1$, and $A= A_1 \oplus A_2$ be an element of $\mathrm{SL}(2n,\C)$ such that $A_1 \in \mathrm{GL}(n,\C)$ and $A_2 \in \mathrm{GL}(n,\C)$ have eigenvalues $\lambda$ and $\lambda^{-1}$, respectively. If  $g \in \mathrm{GL}(n,\C)$ be  an involution   such that  $gAg^{-1} =A^{-1}$,  then  $\mathrm{det}(g) = (-1)^n$
\end{proposition}
\begin{proof} 
	Up to conjugacy, we can assume that $ A_1$ and $ A_2$  are in Jordan form. Since $A= A_1 \oplus A_2$ is reversible,  \thmref{thm-reversible-SL(n, C)} implies that  $ A_1$ and $ A_2$  have the same Jordan structure; see \defref{def-jordan-structure-partition}. Let $A_{W}$ denote the Weyr  form of $A$ such that $A= (A_1)_{W} \oplus (A_2)_{W}$. Then $A_{W} = \tau A \tau^{-1}$ for some $\tau \in \mathrm{SL}(2n,\C)$. Moreover, the Weyr forms $ (A_1)_{W}$ and $ (A_2)_{W}$  have the same Weyr structure, say $(n_1,n_2,\dots,n_r)$; see \defref{def-basic-Weyr-block}. Then the $(i,j)$-th blocks of the Weyr forms  $ (A_1)_{W}$ and $ (A_2)_{W}$ are given by
	$$ ((A_1)_{W})_{i,j} = {\small \begin{cases}
			\lambda	\mathrm{I}_{n_i} & \text{if $j=i$ }\\
			\mathrm{I}_{n_i, n_{i +1}} & \text{if $j =i+1$}\\
			\mathrm{O}_{n_i \times n_{j}}  & \text{otherwise}
	\end{cases}}, \text{ and }  ((A_2)_{W})_{i,j} = {\small  \begin{cases}
			\lambda^{-1}	\mathrm{I}_{n_i} & \text{if $j=i$ }\\
			\mathrm{I}_{n_i, n_{i +1}} & \text{if $j =i+1$}\\
			\mathrm{O}_{n_i \times n_{j}}  & \text{otherwise}
	\end{cases}},$$
	where $1\leq i,  j \leq r$.
	Moreover, the $(i,j)$-th block of the upper triangular block matrix $(A_1)_{W}^{-1}	$ can be written as follows
	$$ 	((A_1)_{W})^{-1}_{i,j}=  {\small\begin{cases}
			\lambda^{-1}\mathrm{I}_{n_i}  & \text{if $j=i$ }\\
			(-1)^k \lambda^{-(k+1)} \mathrm{I}_{n_i, n_{i +k}} & \text{if $j =i+k$}, 1 \leq k \leq r- i\\
			\mathrm{O}_{n_i \times n_{j}}  & \text{otherwise}
		\end{cases}, \text{where $1\leq i,  j \leq r$.}}$$
	Similarly, we can write
	$$ ((A_2)_{W})^{-1}_{i,j} = {\small  \begin{cases}
			\lambda\mathrm{I}_{n_i}  & \text{$j=i$ }\\
			(-1)^k \lambda^{k+1} \mathrm{I}_{n_i, n_{i +k}} & \text{$j =i+k$}, 1 \leq k \leq r- i\\
			\mathrm{O}_{n_i \times n_{j}}  & \text{otherwise}
		\end{cases}, \text{where $1\leq i,  j \leq r$.}}$$
	Consider $\Omega_{W} =  \begin{psmallmatrix}
		& \Omega_{1} \\
		\Omega_{2}	& 
	\end{psmallmatrix} $ in   $\mathrm{GL}(2n,\C)$ such that  $\Omega_1= [ X_{i,j} ]_{1 \leq i,j \leq r} \in \mathrm{GL}(n,\C)$ and $\Omega_2= [ Y_{i,j} ]_{1 \leq i,j \leq r} \in \mathrm{GL}(n,\C)$   are defined as follows
	{\small 	\begin{align*}
			X_{i,j} &=\begin{cases}
				\mathrm{O}_{n_i \times n_{j}} & \text{if $ j<i $}\\
				\mathrm{O}_{n_i \times n_{j}}  & \text{if $j =r,  i  \neq r$}\\
				(-1)^{r-i} \, \lambda^{-2(r-i)} \, \mathrm{I}_{n_i} & \text{if $j=i$ }\\
				(-1)^{r-i} \,  \binom{r-i-1}{j-i}  \,  \lambda^{-2r+i+j}  \,  \mathrm{I}_{n_i, n_{j}}  & \text{if $ i<j, j  \neq r$ }
			\end{cases}, \hbox{ and}
			\\	Y_{i,j} &=\begin{cases}
				\mathrm{O}_{n_i \times n_{j}} & \text{if $ j<i $}\\
				\mathrm{O}_{n_i \times n_{j}}  & \text{if $j =r,  i  \neq r$}\\
				(-1)^{r-i} \, \lambda^{2(r-i)}\, \mathrm{I}_{n_i} & \text{if $j=i$ }\\
				(-1)^{r-i} \,  \binom{r-i-1}{j-i}  \, \lambda^{2r-i-j}  \,  \mathrm{I}_{n_i, n_{j}}  & \text{if $ i<j , j  \neq r$ }
			\end{cases},
	\end{align*}}
	where  $\binom{r-i-1}{j-i} $ denotes the binomial coefficients. Then  by using a similar argument as in \lemref{lem-conjugacy-reverser-unipotent}, we have 
	$$\Omega_{W} A_{W} \Omega_{W}^{-1}=A_{W}^{-1}.$$
	Let $ f \in \mathrm{GL}(2n,\C)$ satisfies $fA=Af$. Since $A = A_1 \oplus A_2$ such that $A_1$ and $A_2$ have no common eigenvalues, \cite[Proposition 3.1.1]{COV} implies that 
	$f =  \begin{psmallmatrix}
		B_1&  \\
		&  B_2
	\end{psmallmatrix},$
where $B_1,B_2  \in \mathrm{GL}(n,\C)$  such that $B_1 A_1=A_1 B_1$ and $B_2A_2=A_2 B_2$. Moreover, since $A_1$ and $A_2$ are basic Weyr matrices with the same Weyr structure,  \propref{prop-centralizer-basic-Weyr-block} implies that $B_1$ and $B_2$ are block upper triangular matrices with the same block structure $(n_1,n_2,\dots,n_r)$. 
	
Consider $h = \tau g \tau^{-1}$ in $\mathrm{GL}(2n,\C)$, where  $\tau \in \mathrm{SL}(2n,\C)$ such that $A_{W} = \tau A \tau^{-1}$.  Since $g$ is  an involution in $\mathrm{GL}(2n,\C)$  such that  $gAg^{-1} =A^{-1}$, it follows that $h$ is an involution in $\mathrm{GL}(2n,\C)$ such that 
	\begin{equation}\label{eq-relation-reverser-Weyr-Jordan-non-unipotent}
		h A_{W} h^{-1} =A_{W}^{-1}, \hbox{ and } \mathrm{det}(h) =\mathrm{det}(g).
	\end{equation}
	Since the set of reversers of $A_{W}$  is a right coset of the centraliser of $A_{W}$, we have $h =  f  \,\Omega_{W}.$ Therefore, 
	$h =   \begin{psmallmatrix}
		B_1&  \\
		&  B_2
	\end{psmallmatrix} \begin{psmallmatrix}
		& \Omega_{1} \\
		\Omega_{2}	& 
	\end{psmallmatrix} = \begin{psmallmatrix}
		& B_1 \Omega_{1} \\
		B_2\Omega_{2}	& 
	\end{psmallmatrix}.$
	This implies that $$\mathrm{det}(h) = \mathrm{det} \begin{psmallmatrix}
		B_1 \Omega_{1} &\\
		&	B_2\Omega_{2}	 
	\end{psmallmatrix} \mathrm{det} \begin{psmallmatrix}
		& \mathrm{I}_n \\
		\mathrm{I}_n	& 
	\end{psmallmatrix}= (-1)^n  \, \mathrm{det} \begin{psmallmatrix}
		B_1 \Omega_{1} &\\
		&	B_2 \Omega_{2}	 
	\end{psmallmatrix}.$$
	Since $h$ is an involution, we get  $  \mathrm{det} \begin{psmallmatrix}
		B_1 \Omega_{1} &\\
		&	B_2 \Omega_{2}	 
	\end{psmallmatrix} = \mathrm{det}(B_1 \Omega_{1})\, \mathrm{det}(B_2 \Omega_{2}) = 1.$ Therefore, $\mathrm{det}(h) = (-1)^n $. The proof now follows from Equation \eqref{eq-relation-reverser-Weyr-Jordan-non-unipotent}.
\end{proof}

The following result generalises \propref{prop-str-rev-subform-SL}\eqref{part-2-prop-str-rev-subform-SL}.
\begin{proposition}\label{prop-str-rev-lambda-non-unipotent}
Let $\lambda \neq \pm1$, and $A= A_1 \oplus A_2$ be a reversible element of $\mathrm{SL}(2n,\C)$ such that all eigenvalues of $A_1 \in \mathrm{GL}(n,\C)$ and $A_2 \in \mathrm{GL}(n,\C)$ are $\lambda$ and $\lambda^{-1}$, respectively. Then $A$ is strongly reversible if and only if $n$ is even.
\end{proposition}
\begin{proof}
If	$A$ is strongly reversible in $\mathrm{SL}(2n,\C)$,  then there exists an involution $g$  in $\mathrm{SL}(n,\C)$  such that  $gAg^{-1} =A^{-1}$. Using \propref{prop-det-rev-inv-non-unipotent-Weyr}, we have $\det{g} = (-1)^n =1$. This implies that $n$ is even.
	
Conversely, let $n$ be even. Since $A$ is reversible, using \thmref{thm-reversible-SL(n, C)}, we can  partition the blocks in the Jordan form of $A$  into pairs $ \{\mathrm{J}(\lambda, r),\mathrm{J}(\lambda^{-1}, r)\} $. Since $n$ is even, there are an even number of pairs $ \{\mathrm{J}(\lambda, r),\mathrm{J}(\lambda^{-1}, r)\} $ with $r$ is odd.  Therefore,  using \remref{rem-construction-rev-inv} and \propref{prop-str-rev-subform-SL}, we can construct a suitable involution $g $ in $ \mathrm{SL}(n,\C) $ such that $ gAg^{-1} =A^{-1}$. Therefore, $A$ is strongly reversible in $ \mathrm{SL}(n,\C) $. This proves the result.
\end{proof}

\subsection{Proof of \thmref{thr-main-strong-rev-SL(n, C)}}
Up to conjugacy, we can assume that $A$ is in Jordan form. Since $A \in \mathrm{SL}(n,\C)$ is reversible, \thmref{thm-reversible-SL(n, C)} implies that the Jordan blocks in the Jordan form $A$ can be partitioned into singletons $\{\mathrm{J}(\mu, k)\}$ or pairs $ \{\mathrm{J}(\lambda, m),\mathrm{J}(\lambda^{-1}, m)\} $, where $  \mu \in \{ -1,+1\}$, $\lambda \notin \{ -1, +1\}$. Let $ \lambda_1, \lambda_2,\dots,\lambda_\ell$ and their inverses $ \lambda_1^{-1}, \lambda_2^{-1},\dots, \lambda_\ell^{-1}$ be distinct eigenvalues of $A$ such that  $\lambda_i \notin \{ -1,+1\} $ and $\lambda_j \neq \lambda_{j'}$ or $\lambda_{j'}^{-1}$ for all $j \neq j'$, where $1 \leq i, j,j' \leq \ell$. Furthermore, suppose that both $\lambda_i$ and $\lambda_i^{-1}$ have multiplicity $m_i$ for all $1 \leq  i \leq \ell$. Therefore, up to conjugacy, we can assume that
\begin{equation} \label{eq-block-Jordan-matrices-str-rev-general}
	A = P \oplus Q \oplus \begin{psmallmatrix}
		R_{1} &  \\
		&  R'_{1}
	\end{psmallmatrix} \oplus \begin{psmallmatrix}
		R_{2} &  \\
		&  R'_{2}
	\end{psmallmatrix}  \oplus \cdots  \oplus \begin{psmallmatrix}
		R_{\ell} &  \\
		&  R'_{\ell}
	\end{psmallmatrix},
\end{equation}
such that $P \in \mathrm{GL}(p,\C)$, $Q \in \mathrm{GL}(q,\C)$, $R_i\in \mathrm{GL}(m_i,\C)$ and $R'_i  \in \mathrm{GL}(m_i,\C)$ are Jordan matrices corresponding to eigenvalues $+1$, $-1$,  $\lambda_i$ and $\lambda_i^{-1}$, respectively, where $1 \leq  i \leq \ell$.  Then the Weyr form $A_{W}$ corresponding to the Jordan form $A$ can be given by
\begin{equation} \label{eq-block-Weyr-matrices-str-rev-general}
	A_{W} = P_{W}  \,  \oplus Q_{W} \,  \oplus \begin{psmallmatrix}
		(R_{1})_{W}  &  \\
		&  (R'_{1})_{W} 
	\end{psmallmatrix} \oplus \begin{psmallmatrix}
		(R_{2})_{W}  &  \\
		&  (R'_{2})_{W} 
	\end{psmallmatrix}  \oplus \cdots  \oplus \begin{psmallmatrix}
		(R_{\ell})_{W}  &  \\
		&  (R'_{\ell})_{W} 
	\end{psmallmatrix},
\end{equation}
such that $P_{W}, Q_{W},(R_{i})_{W} $ and $(R'_{i})_{W}$ are basic Weyr matrices corresponding to Jordan matrices  $P,Q,R_{i}$ and $R'_{i}$, respectively, where $1 \leq  i \leq \ell$.  Since diagonal block matrices in Weyr form $A_{W}$ do not have common eigenvalues,  \cite[Proposition 3.1.1]{COV} implies  that any $f \in \mathrm{GL}(n,\C)$ commuting with $A_{W}$  has the following form
\begin{equation} \label{eq-commutator-block-Weyr-matrices-str-rev-general}
	f =B \oplus C \oplus  \begin{psmallmatrix}
		D_{1} &  \\
		&  D'_{1}
	\end{psmallmatrix} \oplus \begin{psmallmatrix}
		D_{2} &  \\
		&  D'_{2}
	\end{psmallmatrix}  \oplus \cdots  \oplus \begin{psmallmatrix}
		D_{\ell} &  \\
		&  D'_{\ell}
	\end{psmallmatrix},
\end{equation}
such that $B  \in \mathrm{GL}(p,\C), C  \in \mathrm{GL}(q,\C), D_i \in \mathrm{GL}(m_i,\C)$ and $D'_i \in \mathrm{M}(m_i,\C)$, and satisfies  $ B P_{W} = P_{W}B, \, C Q_{W}= Q_{W}  C, \,	D_{	i} 	(R_{i})_{W}  =	(R_{i})_{W}	D_{i} $ and $	D'_{i} (R'_{i})_{W} =(R'_{i})_{W}	D'_{i}$, respectively,  where $1 \leq i \leq \ell$.

Suppose that $A$ is strongly reversible in $ \mathrm{SL}(n,\C)$. Since $A_{W}$ is the Weyr form of $A$, there exists $\tau \in \mathrm{SL}(n,\C)$ such that $A_{W} = \tau A\tau^{-1}$. Therefore, $ A_{W}$ is strongly reversible in $ \mathrm{SL}(n,\C)$. This implies that there exists an involution $g \in \mathrm{SL}(n,\C)$ such that $g A_{W} g^{-1} =A_{W}^{-1}$. 
Note that  using a similar argument as  in \propref{prop-unipotent-strong-rev-SL(n, C)}, \propref{prop-strong-rev-SL(n, C)-eigenvalue-minus-1}  and \propref{prop-str-rev-lambda-non-unipotent}, we can find a suitable  reverser of $A_{W}$ given in Equation \eqref{eq-block-Weyr-matrices-str-rev-general}.  Since the set of  reversers is a right coset of the centraliser, it follows that $g$ has the following form
\begin{equation}\label{eq-reverser-block-Weyr-matrices-str-rev-general}
	g=  \alpha \oplus \beta \oplus g_1 \oplus \cdots  \oplus g_\ell,
\end{equation}
such that $\alpha  \in \mathrm{GL}(p,\C), \beta  \in \mathrm{GL}(q,\C),$ and $ g_i \in \mathrm{GL}(2m_i,\C)$, and satisfies  $ \alpha P_{W} \alpha^{-1}= P_{W}^{-1}, \, \beta Q_{W} \beta^{-1}=Q_{W}^{-1},$  and $ g_i \begin{psmallmatrix}
	(R_{i})_{W} &  \\
	&  	(R'_{i})_{W}
\end{psmallmatrix} g_i^{-1} = \begin{psmallmatrix}
(R_{i})_{W} &  \\
&  	(R'_{i})_{W}
\end{psmallmatrix} ^{-1}$, respectively, where $1 \leq i \leq t$.

Let $s$ denote the number of Jordan blocks of the form $\{\mathrm{J}(\mu, 4k+2)\}$, $\mu \in \{-1, + 1\}$, and $t$ denote the number of pairs of the form $ \{ \mathrm{J}(\lambda, 2m+1),\mathrm{J}(\lambda^{-1}, 2m+1)\} $, $ \lambda \notin \{- 1, +1\}$, 
in the Jordan form $A$. Note that if  condition $(\ref{cond-1-thr-strong-SL(n, C)})$ of \thmref{thr-main-strong-rev-SL(n, C)} holds or $s=t=0$, then using \remref{rem-construction-rev-inv} and  \propref{prop-str-rev-subform-SL},  we can construct an involution  $h \in \mathrm{SL}(n,\C)$ such that $hAh^{-1} =A^{-1}$.

Now, assume that $A$ does not satisfy condition $(\ref{cond-1-thr-strong-SL(n, C)})$ of \thmref{thr-main-strong-rev-SL(n, C)} (i.e., both ${\O}_{\d(p)}$ and  ${\O}_{\d(q)}$  are empty) and $s+t \neq 0$, otherwise we are done.
Using  \propref{prop-det-reverser-inv-unipotent-SL(n, C)}, \propref{prop-det-reverser-inv-SL(n, C)-eigenvalue-minus-1}, \propref{prop-det-rev-inv-non-unipotent-Weyr}, and Equation \eqref{eq-reverser-block-Weyr-matrices-str-rev-general}, we have 
$$\mathrm{det}(g) = (-1)^{\lvert \E_{{\d}(p)}^2 \rvert}  \, (-1)^{\lvert \E_{{\d}(q)}^2 \rvert}  \, \prod_{i=1}^{\ell}(-1)^{m_i} = (-1)^{\lvert \E_{{\d}(p)}^2 \rvert+ \lvert \E_{{\d}(q)}^2 \rvert + \sum_{i=1}^{\ell}m_i},$$
where $p+q + 2  \sum_{i=1}^{\ell}m_i =n$.  Since  $g \in \mathrm{SL}(n,\C)$, 
$\mathrm{det}(g) =(-1)^{\lvert \E_{{\d}(p)}^2 \rvert+ \lvert \E_{{\d}(q)}^2 \rvert +  \sum_{i=1}^{\ell}m_i} =1.$
This implies that $$ \lvert \E_{{\d}(p)}^2 \rvert+ \lvert \E_{{\d}(q)}^2 \rvert +  \sum_{i=1}^{\ell}m_i \hbox{ is even}.$$
Since $s =  \lvert \E_{{\d}(p)}^2 \rvert+ \lvert \E_{{\d}(q)}^2 \rvert$, we get $$s +  \sum_{i=1}^{\ell}m_i \hbox{ is even}.$$
This implies that $s+t$ is even, and thus $A$ satisfies condition $(\ref{cond-2-thr-strong-SL(n, C)})$ of \thmref{thr-main-strong-rev-SL(n, C)}. Therefore, the forward direction of the theorem is proven.

Conversely, up to conjugacy, we can assume that $A$ is in Jordan form as given by Equation \eqref{eq-block-Jordan-matrices-str-rev-general}. If $A$ satisfies at least one of the conditions $(\ref{cond-1-thr-strong-SL(n, C)})$ and $(\ref{cond-2-thr-strong-SL(n, C)})$ of \thmref{thr-main-strong-rev-SL(n, C)}, then using \remref{rem-construction-rev-inv} and  \propref{prop-str-rev-subform-SL}, we can construct a suitable involution $g$ in $\mathrm{SL}(n,\C)$ such that $gAg^{-1} = A^{-1}$. Therefore, $A$ is strongly reversible in $\mathrm{SL}(n,\C)$. This completes the proof. 
\qed

\subsection*{Acknowledgment} 
The authors thank the referee for the careful reading, helpful comments, and suggestions.
Gongopadhyay is partially supported by the SERB core research grant CRG/2022/003680. Lohan acknowledges the financial support from the IIT Kanpur Postdoctoral Fellowship. Maity is partially supported by the Seed Grant IISERBPR/RD/OO/2024/23.

\end{document}